\documentclass[10pt]{amsart}
\usepackage[margin=3cm]{geometry}
\usepackage{amsmath,amssymb,amsfonts,amsthm}
\usepackage{mathtools}
\usepackage{verbatim}
\usepackage{graphicx}
\usepackage{color}
\usepackage{enumerate}
\usepackage{tikz}
\usetikzlibrary{calc}
\usepackage[T1]{fontenc}
\usepackage{hyperref,cleveref}
\usepackage{microtype}
\hypersetup{
    pdftitle={Conformal dimension bounds, Pontryagin sphere boundaries, and
  algebraic fibering of right-angled Coxeter groups},    
    pdfauthor={Christopher Cashen, Pallavi Dani, Kevin Schreve, and
      Emily Stark},     
    pdfkeywords={Right-angled Coxeter groups, hyperbolic groups, quasi-isometry classification, conformal dimension, Gromov boundary, Pontryagin sphere, algebraic fibering, virtual cohomological dimension, round trees}, 
    colorlinks=true,       
    linkcolor=black,          
    citecolor=black,        
    filecolor=black,      
    urlcolor=black           
  }

\numberwithin{equation}{section}
\numberwithin{figure}{section}

\theoremstyle{plain}
\newtheorem*{reftheorem*}{Theorem}
\newtheorem*{refcor*}{Corollary}
\newtheorem{thm}{Theorem}[section]
\newtheorem{lemma}[thm]{Lemma}
\newtheorem{prop}[thm]{Proposition}
\newtheorem{corollary}[thm]{Corollary}
\theoremstyle{definition}
\newtheorem{defn}[thm]{Definition}

\newtheorem{remark}[thm]{Remark}

\newtheorem{notation}[thm]{Notation}
\newtheorem{question}[thm]{Question}
\newtheorem{construction}[thm]{Construction}

\newcommand{\Z}{\mathbb{Z}}

\newcommand{\N}{\mathbb{N}}

\newcommand{\F}{\mathbb{F}}

\newcommand{\cP}{\mathcal{P}}

\newcommand{\cT}{\mathcal{T}}

\newcommand{\G}{\Gamma}

\newcommand{\CAT}{\operatorname{CAT}}

\newcommand{\Aut}{\operatorname{Aut}}

\newcommand{\Confdim}{\operatorname{Confdim}}

\newcommand{\St}{\operatorname{St}}
\newcommand{\Lk}{\operatorname{Lk}}

\newcommand{\la}{\langle}
\newcommand{\ra}{\rangle}
\newcommand{\bdry}{\partial}
\newcommand{\p}{\partial}

\newcommand{\from}{\colon\thinspace}

\newcommand{\bt}{\boldsymbol{t}}
\newcommand{\LeviP}{\mathcal{P}}
\newcommand{\LeviA}{\mathcal{A}}
\newcommand{\LeviB}{\mathcal{B}}
\newcommand{\LeviT}{\mathcal{T}}
\newcommand{\td}{\mathrm{TD}}

\DeclareRobustCommand{\edge}{%
  \mathrel{%
    \mathchoice
      {\tikz[baseline=-0.6ex,line width=0.45pt]\draw (0,0)--(.9em,0);}
      {\tikz[baseline=-0.55ex,line width=0.45pt]\draw (0,0)--(.8em,0);}
      {\tikz[baseline=-0.35ex,line width=0.35pt]\draw (0,0)--(0.6em,0);}
      {\tikz[baseline=-0.30ex,line width=0.30pt]\draw (0,0)--(0.5em,0);} 
  }%
}

  \newif\ifshowfigurealttext
\showfigurealttextfalse 
\newcommand{\figurealttext}[1]{%
  \ifshowfigurealttext
    \par\smallskip
    {\small\noindent\textbf{Alt text:} #1\par}%
  \fi
}

\usepackage{fancyhdr}
\fancypagestyle{tp}{
\fancyhead{}
\fancyhead[C]{Author's accepted manuscript. The published version,
  subject to  editorial changes, appears open access at
International Mathematics Research Notices
IMRN 2026 (2026), no.~14, rnag144.
\href{https://doi.org/10.1093/imrn/rnag144}{\texttt{\textsc{DOI:}10.1093/imrn/rnag144}}
}
}

\title[Conformal dimension, Pontryagin boundaries, and algebraic fibering of RACGs]{Conformal dimension bounds, Pontryagin sphere boundaries, and  algebraic fibering of right-angled Coxeter groups}
\date{June 12, 2025}


\author[Cashen]{Christopher H.\ Cashen} 
\address{TU Wien\\  Institute of Discrete Mathematics and Geometry\\Wiedener
  Hauptstrasse~8-10\\1040 Vienna \\Austria\\\href{https://orcid.org/0000-0002-6340-469X}{\includegraphics[scale=.75]{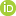}~0000-0002-6340-469X}}
\email{christopher.cashen@tuwien.ac.at}
\author[Dani]{Pallavi Dani}
\address{Department of Mathematics, Louisiana State University, Baton
  Rouge, LA 70803–4918, USA\\
  \phantom{fill up the line}
 \href{https://orcid.org/0000-0001-7653-2618}{\includegraphics[scale=.75]{ORCID-iD_icon-16x16}~0000\nobreakdash-0001\nobreakdash-7653\nobreakdash-2618}}
\email{pdani@math.lsu.edu}
\author[Schreve]{Kevin Schreve}
\address{Department of Mathematics\\  Louisiana State University\\Baton Rouge\\ Louisiana \\USA\\}
\email{kschreve@lsu.edu}
\author[Stark]{Emily Stark}
\address{Department of Mathematics\\  Wesleyan University\\ Middletown \\ Connecticut \\USA\\}
\email{estark@wesleyan.edu}

\begin{document}
\begin{abstract}
We introduce a graph-theoretic condition, called $(n,m)$--branching, that ensures a combinatorial round tree with controlled branching parameters can be quasi-isometrically embedded in the Davis complex of the right-angled Coxeter group defined by the graph. This construction yields a lower bound on the conformal dimension of the boundary of such a hyperbolic group. We exhibit numerous families of graphs with this property, including many 1-dimensional spherical buildings.

We prove an embedding result, showing that under mild hypotheses a flag-no-square graph embeds as an induced subgraph in a flag-no-square triangulation of a closed surface.
We use this to embed our branching graphs into graphs presenting hyperbolic right-angled Coxeter groups with Pontryagin sphere boundary.  We conclude there are examples of such groups with conformal dimension tending to infinity, and hence, there are infinitely many quasi-isometry classes within this family.

We use conformal dimension to show that recent work of Lafont--Minemyer--Sorcar--Stover--Wells can be upgraded to conclude that for every $n \geq 2$ there exist infinitely many quasi-isometry classes of hyperbolic right-angled Coxeter groups that virtually algebraically fiber and have virtual cohomological dimension $n$.

\end{abstract}
\maketitle
\thispagestyle{tp}

\section{Introduction}\label{sec:introduction}

The goal of this paper is to exhibit infinitely many quasi-isometry classes within families of right-angled Coxeter groups that satisfy properties of interest.
We define conditions on the presentation graph of the Coxeter group that imply the group has a desired property, and then we show that the graph conditions are loose enough to allow rich families of subgraphs, from which we deduce quasi-isometry invariants via conformal dimension. 
We focus on the following two properties:
\begin{enumerate}[1. ]
\item Hyperbolic with Pontryagin sphere boundary.
  \item Hyperbolic, virtually algebraically fibered, and of any given virtual cohomological dimension.
  \end{enumerate}
  To differentiate quasi-isometry types, we analyze the Gromov boundaries of the groups. 
  The homeomorphism type of the boundary of a hyperbolic group is a quasi-isometry invariant, but it is too coarse for our purposes.
  Metric information refines this invariant: the quasisymmetry homeomorphism type of the boundary is a complete quasi-isometry invariant.
Pansu's conformal dimension \cite{pansu}  is a powerful quasisymmetry invariant  that reveals the fractal behavior among deformations of a metric space. 

In this paper we develop techniques for bounding conformal dimensions of boundaries of hyperbolic right-angled Coxeter groups. 
We show that the families described above contain sequences of groups whose boundary conformal dimensions grow without bound.
In particular, this implies that there are infinitely many quasi-isometry types of groups within the sequence. 

\subsection{Pontryagin sphere boundaries}
Pansu's original work computed the conformal dimension of the boundaries of the rank-one symmetric spaces, which are each homeomorphic to the sphere.
A primary application of the work in this paper gives lower bounds on the conformal dimension of Pontryagin sphere boundaries.
The Pontryagin sphere is an inverse limit of surfaces and arises naturally as the boundary of 3--dimensional hyperbolic pseudo-manifolds, where the link of each point is either the sphere or a higher genus surface.
Right-angled Coxeter groups provide an explicit source of a vast family of groups with Pontryagin sphere boundary; Fischer~\cite{fischer} proved the boundary of each right-angled Coxeter group defined by a flag-no-square triangulation of an orientable higher-genus surface is homeomorphic to the Pontryagin sphere. 

        \begin{reftheorem*}[\Cref{thm:Psphere_lower_bound}] 
          There exists a family of hyperbolic right-angled Coxeter groups $W_i$ such that $\bdry W_i$ is the Pontryagin sphere and the conformal dimension of $\bdry W_i$ tends to infinity as $i \to \infty$.
           In particular, there are infinitely many quasi-isometry classes of such groups. 
      \end{reftheorem*}

      We prove \Cref{thm:Psphere_lower_bound} by embedding into our groups certain metric spaces whose boundaries have known conformal dimension.
      These spaces, \emph{round trees}, were introduced by Gromov~\cite[Section 7.C3]{gromov93} for this purpose.
      They have been employed by Bourdon~\cite{bourdon95} in the setting of right-angled Fuchsian buildings and by Mackay~\cite{mackay, mackay16} in the setting of certain small cancellation groups and random groups. Mackay~\cite{mackay16} subsequently developed a combinatorial version of round trees, useful in producing lower bounds on the conformal dimension of the boundary of hyperbolic and relatively hyperbolic groups~\cite{frost, fieldguptalymanstark}. See \cite{stark-survey} for background on this technique. 

      We introduce a graph theoretic condition, \emph{$(n,m)$--branching} (\Cref{defn:n6_branching}), that ensures that a combinatorial round tree with prescribed branching parameters quasi-isometrically embeds in the Davis complex of the right-angled Coxeter group defined by the graph.
      Mackay's bounds then yield:
    \begin{reftheorem*}[\Cref{thm:confdim_lower_bounds}] 
        If $\Gamma$ is a graph with
        $(n,m)$--branching, then the right-angled Coxeter
        group $W_\Gamma$ is hyperbolic and the conformal dimension of the boundary satisfies:
            \[ \Confdim(\p W_\Gamma) \geq 1 + \frac{\log n}{\log 3m-7} \]
    \end{reftheorem*}

    In \Cref{sec:branching_from_projective_planes}, we exhibit numerous families of graphs with the $(n,m)$-branching condition.
    Examples naturally arise via Levi incidence graphs of finite incidence structures.
    In \Cref{sec:Pontryagin}, we show such graphs can be embedded as induced subgraphs in flag-no-square triangulations of a surface, which yields \Cref{thm:Psphere_lower_bound}. 

    Field, Gupta, Lyman, and Stark \cite[Theorem~D]{fieldguptalymanstark} constructed infinitely many quasi-isometry classes of hyperbolic Coxeter groups with Pontryagin sphere boundary, also using conformal dimension techniques.
    However, that construction used specific features of a family of non-right-angled Coxeter groups that are not adaptable to the right-angled case.

\subsection{Algebraic Fibrations}
A group \emph{algebraically fibers} if it surjects onto $\Z$ with finitely generated kernel. Algebraic fibrations are related to fibering of $3$--manifolds, the Bieri-Neumann-Strebel invariants, and have recently been connected to $L^2$--homology by work of Kielak \cite{Kie20}.
Lafont, Minemyer, Sorcar, Stover, and Wells \cite{LMSSW-fibering} recently provided the first examples of high-dimensional right-angled Coxeter groups that virtually algebraically fiber:
    \begin{reftheorem*}[{\cite[Main Theorem]{LMSSW-fibering}}]\label{LMSSW}
        For every $n \geq 2$ there exist infinitely many isomorphism classes of hyperbolic right-angled Coxeter groups that virtually algebraically fiber and have virtual cohomological dimension~$n$. 
    \end{reftheorem*}
We use a generalization of their construction of examples, but we approach the problem from a different perspective that allows us to use an additional suite of tools.
We establish fibering for a wider family of examples, with growing conformal dimension but fixed virtual cohomological dimension.

    \begin{reftheorem*}[\Cref{thm:QI_of_fibering_examples}]
        For every $n \geq 2$ there exist hyperbolic right-angled Coxeter groups $W_i$ that virtually algebraically fiber, have virtual cohomological dimension~$n$, and have $\Confdim(\bdry W_i) \to \infty$. In particular, there are infinitely many quasi-isometry classes of such groups. 
    \end{reftheorem*}

    \subsection{Further questions}
    Topological dimension provides a lower bound for conformal dimension.
    The topological dimension of the Pontryagin sphere is two, yet this space is a nowhere-planar fractal, so we are very interested in the following:
\begin{question}
    Are there hyperbolic groups $\Gamma$ (perhaps not necessarily right-angled Coxeter) with Pontryagin sphere boundary whose conformal dimension equals 2? If not, are there examples with conformal dimension arbitrarily close to 2?
\end{question}

Douba, Lee, Marquis, and Ruffoni~\cite{DouLeeMar26} recently constructed an example of a hyperbolic right-angled Coxeter group $W_\Gamma$ with Pontryagin sphere boundary that has a convex cocompact isometric action on $\mathbb{H}^4$, so $\Confdim(\bdry W_\Gamma) < 3$ \cite[Proposition 8.3, Proposition 8.4]{CarMac22}.
This is the best upper bound of which we are aware.

In general, establishing upper bounds on conformal dimension is difficult.
It would be interesting to have some local control of this quantity.
For example, our construction of branching graphs to produce high conformal dimension necessarily causes the link topology to explode (see \Cref{lem:branching_implies_nonplanar}).
One might wonder if containing the link topology constrains the conformal dimension of the boundary:

\begin{question} \label{question:upper_bound}
Does there exist a bound $d(g)$ such that if $W_\Gamma$ is a hyperbolic right-angled Coxeter group whose nerve is a closed surface of genus $g$, then the conformal dimension of $\p W_\Gamma$ is at most $d(g)$?
\end{question}

Bourdon and Kleiner obtain upper bounds on conformal dimension via $\ell_p$--cohomology, but their conditions for Coxeter groups \cite[Corollary~8.1]{bourdonkleiner15} force the group to not be right-angled and to have 1--dimensional nerve. In particular, their results do not apply to \Cref{question:upper_bound}.

\medskip

    The work of Lafont--Minemyer--Sorcar--Stover--Wells and the examples provided in this paper exhibit families of high-dimensional hyperbolic groups that virtually algebraically fiber and that fall into infinitely many quasi-isometry classes. Understanding the homeomorphism type of the boundaries of these groups would be interesting and lies in the context of the following question. 

    \begin{question}
        Describe the possible boundaries of hyperbolic groups of virtual cohomological dimension~$n$ that virtually algebraically fiber. 
    \end{question}

    \subsection*{Acknowledgments}
This material is based upon work supported by the U.S.\  National Science Foundation under Award Numbers  DMS-2203325, DMS-2204339, DMS-2407104, and DMS-2505290.
 This research was funded in part by the Austrian Science
Fund (FWF) \href{https://doi.org/10.55776/PAT7799924}{10.55776/PAT7799924} and the Austria-Slovakia research cooperation grant ``Constructions of
expanders and extremal graphs'', OeAD WTZ SK 14/2024 and APVV
SK-AT-23-0019.

The authors are grateful to the Erwin Schr\"odinger Institute, which hosted our Research in Teams project ``Rigidity in Coxeter groups'' in July 2023, of which this paper is a product. We thank ESI for their generous hospitality during our stay.

\section{Preliminaries}\label{sec:preliminaries}

   Throughout the paper we let $\N := \Z_{\geq 1}$.
   The graphs we consider are simplicial, so an edge between distinct vertices $x$ and $y$ is unambiguously denoted $x \edge y$.

    \subsection{Right-angled Coxeter groups and cell complexes}

    \begin{defn}
      Let $\Gamma$ be a finite simplicial graph with vertex set $V\Gamma$ and edge set $E \Gamma$.
      The \emph{ right-angled Coxeter group} $W_{\Gamma}$ is the group with presentation:
            \[ W_{\Gamma} := \la \, v \in V\Gamma \,\, | \,\, v^2=1 \textrm{ for all } v \in V\Gamma, [v,w]=1 \textrm{ if and only if } v\edge w \in E\Gamma\ \ra \]
    \end{defn}

    A right-angled Coxeter group $W_{\Gamma}$ acts properly and cocompactly by isometries on its associated \emph{Davis complex} $X_{\Gamma}$.
    The Davis complex is a $\CAT(0)$ cube complex whose 1--skeleton is the Cayley graph for $W_{\Gamma}$, with bigons coming from the relations $v^2=1$ collapsed. 
    Via the identification with the Cayley graph, we consider edges in the Davis complex to be labeled by generators of the group, or, equivalently, by vertices of $\Gamma$.
    If $\Gamma$ has no triangles, then the link of every vertex in the Davis complex is isomorphic to $\Gamma$. We refer the reader to \cite{davis, dani-survey} for background. 

    \begin{defn}
        Let $L$ be a finite simplicial complex. 
        \begin{enumerate}[1. ]
            \item A subcomplex $L' \subset L$ is \emph{induced} or \emph{full} if 
            whenever a simplex in $L$ has vertex set in $L'$, the simplex itself is in $L'$.
            \item $L$ is \emph{flag} if any collection of pairwise adjacent vertices spans a simplex.
            \item $L$ is \emph{flag-no-square} if $L$ is flag and has no induced $4$-cycles. 
        \end{enumerate}
    \end{defn}
   
    We utilize the following well-known properties. 

    \begin{lemma} \label{lemma:RACG_props}
        Let $\Gamma$ be a finite simplicial graph. 
        \begin{enumerate}[1. ]
            \item The right-angled Coxeter group $W_{\Gamma}$ is hyperbolic if and only if $\Gamma$ has no induced squares. \cite{moussong}
            \item Let $\Gamma' \subset \Gamma$ be an induced subgraph. Then $W_{\Gamma'}$ is isomorphic to a subgroup of $W_{\Gamma}$  whose Davis complex $X_{\Gamma'}$ is isomorphic to a convex subcomplex of $X_\Gamma$. \cite[Section 4.1]{davis} 
        \end{enumerate}
    \end{lemma}
    
    \begin{notation}
      The cell complexes in this paper have cells that are convex polytopes, either simplices or cubes. Every vertex in these two types of polytope has link a simplex.
      If $X$ is the cell complex and $x$ is a vertex of $X$ then the \emph{link} $\Lk(x)$ of $x$ is the simplicial complex obtained from the links of $x$ in each closed cell properly containing it.
      We let $\Lk_{A}(x)$ denote the link of a vertex $x \in A \subset X$ restricted to a subcomplex $A \subset X$.  If $A'\subset A$ is a subcomplex of a polygonal complex $A$, then the \emph{star} of $A'$ in $A$, denoted $\St(A')$, is the union of all closed cells in $A$ that nontrivially intersect $A'$.

          When $X$ is a simplicial graph there is a canonical bijection between edges containing $x$ and neighbors of $x$, so we also use $\Lk(x)$ to denote the set of neighbors of $x$ and $\St(x)$ to be $\Lk(x)\cup\{x\}$. 
    \end{notation}

\subsection{Boundaries and conformal dimension}

\begin{defn}
Let $X$ be a proper geodesic metric space.
The \emph{visual boundary} of $X$ is denoted $\p X$ and consists of all equivalence classes of geodesic rays, where two rays $\gamma, \gamma'\from[0,\infty) \to X$ are \emph{equivalent} if there exists a constant $C$ such that $d\bigl(\gamma(t), \gamma'(t)\bigr) \leq C$ for all $t \in [0,\infty)$. 
\end{defn}

If $X$ is a $\delta$--hyperbolic proper geodesic metric space, then there is a family of metrics on $\p X$, called \emph{visual metrics}, that induce a natural topology on $\p X$. See \cite{bridsonhaefliger} for background. 

\begin{defn}
Let $(Z,d)$ and $(Z',d')$ be metric spaces. A homeomorphism $f\from Z \to Z'$ is a \emph{quasisymmetry} if there exists a homeomorphism $\phi \from \left[0,\infty\right) \to \left[0,\infty\right)$ such that for distinct $x,y,z \in Z$:
  \[
    \frac{d'\bigl(f(x), f(y)\bigr)}{d'\bigl(f(x), f(z)\bigr)} \le \phi\left( \frac{d(x,y)}{d(x,z)} \right)
  \]
The spaces $(Z,d)$ and $(Z',d')$ are \emph{quasisymmetric} if there exists a quasisymmetry $f\from Z \to Z'$.
\end{defn}

\begin{thm}[\cite{bourdon-flot}, Theorem 1.6.4; \cite{buyaloschroeder}, Theorem 5.2.17] \label{thm_qs} 
  Let $X$ and $X'$ be proper geodesic $\delta$--hyperbolic metric spaces, and let $(\partial X,d_a)$ and $(\partial X',d_{a'})$ denote their visual boundaries equipped with visual metrics.
  A quasi-isometry $f\from X \to X'$ extends to a quasisymmetry 
  $\bdry f\from (\partial X,d_a) \to (\partial X',d_{a'})$.
\end{thm}

    The theorem above implies that the quasisymmetry class of the boundary of a proper geodesic $\delta$--hyperbolic space equipped with a visual metric is a quasi-isometry invariant of the $\delta$--hyperbolic space. This result yields the following quasisymmetry invariant introduced by Pansu~\cite{pansu}. 

    \begin{defn}
      Let $(Z,d)$ be a metric space.
      The \emph{conformal dimension} of $Z$ is the infimal Hausdorff dimension of any metric space $(Z',d')$ quasisymmetric to $(Z,d)$.
    \end{defn}

    See Mackay and Tyson \cite{mackaytyson} for additional background on conformal dimension. Note that the conformal dimension of the boundary of a hyperbolic group equipped with a visual metric is finite, see \cite[Theorem~3.4.4]{mackaytyson}.

    \subsection{Pontryagin spheres}

    The Pontryagin sphere is, roughly speaking, constructed by starting with a closed orientable surface and repeatedly replacing the interior of an embedded disk with a once punctured torus. In particular, it is an inverse limit of surfaces, where the maps between surfaces can be realized as collapsing disjoint handles to disks. More formally, we will use the following equivalent description of Daverman and Thickstun \cite{davermanthickstun}. Here a \emph{figure-eight} is a space homeomorphic to the wedge of two circles.
    
    \begin{defn}
    Suppose $\cP$ is a compact, connected, locally path-connected, separable, and metrizable space. Then $\cP$ is a \emph{Pontryagin sphere} if there exists a countable family $\mathcal{E}$ of pairwise-disjoint figure-eights in $\cP$ such that $\mathcal{E}$ is null and dense in $\cP$ and, for any cofinite subfamily $\mathcal{D}$ of $\mathcal{E}$, the image of $\cP$ under the decomposition map $\cP \to \cP/\mathcal{D}$ is a closed orientable surface.
    \end{defn}

    Jakobsche \cite{MR1113561} showed that the Pontryagin sphere is homogeneous and unique up to homeomorphism. Fischer \cite{fischer} showed that if $L$ is a flag triangulation of a closed, orientable surface of genus $\ge 1$, then the visual boundary of the associated Davis complex is homeomorphic to $\cP$.
    \'Swi\k{a}tkowski \cite{swiatkowski-trees} proved the same statement for the visual boundaries of the universal covers of locally CAT$(-1)$ $3$--dimensional pseudo-manifolds (that are not manifolds).

\section{Conformal dimension and virtual algebraic
  fibering}\label{sec:fibering}

  The aim of this section is to prove an enhancement of a result of Lafont, Minemyer, Sorcar, Stover, and Wells (LMSSW) on virtual fibering of right-angled Coxeter groups \cite[Main Theorem]{LMSSW-fibering}. We are able to quasi-isometrically embed Bourdon's right-angled Fuchsian buildings in a family of examples, and hence we can use his conformal dimension computations in this case. 
  In later sections, we will prove lower bounds on conformal dimension for examples which do not necessarily contain quasi-isometrically embedded right-angled buildings.

  
    \begin{thm} \label{thm:QI_of_fibering_examples}
        For every $n \geq 2$ there exist infinitely many quasi-isometry classes of hyperbolic right-angled Coxeter groups that virtually algebraically fiber and have virtual cohomological dimension~$n$. 
    \end{thm}

    \begin{proof}
      Osajda \cite{Osa13} used a `thickening' operation that takes as input a finite cube complex $X$ and outputs a flag simplicial complex $\mathrm{Th}^1(X)$ with the same vertex set, such that two vertices are joined by an edge if and only if they are contained in a common cube of $X$.
      LMSSW \cite{LMSSW-fibering} generalized this operation by taking as input a finite cube complex $X$ together with a multiplicity function $\mu\from VX\to \mathbb{N}$ and outputting the simplicial complex $\mathrm{Th}^\mu(X)$ whose vertex set is  $\{(v,1),\dots,(v,\mu(v))\mid v\in VX\}$, with a simplex defined by a set of vertices whenever their first coordinates belong to a common cube in $X$.
           Observe that $\mathrm{Th}^1(X)$ is precisely $\mathrm{Th}^\mu(X)$ for $\mu\equiv 1$. 
           We use $\mathrm{Th}^N$ when $\mu\equiv N$.
    It is shown in \cite{LMSSW-fibering}  that if $X$ is 5--large, then so is $\mathrm{Th}^\mu(X)$. 

      A sequence of spaces is constructed by taking $X_1$ to be, say,  a pentagon and, given $X_n$, choosing a $\mu_n$ and taking  $G_{n+1}:=W_{\mathrm{Th}^{\mu_n}(X_n)}$ and $X_{n+1}$ to be a quotient of the Davis complex $\Sigma_{\mathrm{Th}^{\mu_n}(X_n)}$ by a suitably chosen torsion-free finite-index subgroup of $G_{n+1}$. 
      Then repeat.
      The persistence of 5--largeness implies that the $G_n$ are one-ended hyperbolic groups.
      LMSSW choose their $\mu_n$ carefully, depending on the geometry of $X_n$, for purposes to be mentioned shortly, and compute $vcd(G_n)=n$.

We observe that for any $X$ and $\mu$, one can view $W_{{\mathrm{Th}^\mu(X)}}$ as a graph product of finite groups with underlying graph $\mathrm{Th}^1(X)$ and so that the vertex group for $v\in V\mathrm{Th}^1(X)=VX$ is $(\mathbb{Z}/2)^{\mu(v)}$.
It follows that $W_{{\mathrm{Th}^\mu(X)}}$ is a lattice in the automorphism group of a building $\Phi^\mu(X)$ whose underlying Coxeter group is $W_{\mathrm{Th}^1(X)}$ \cite[Theorem~5.1]{Dav98}, and whose thickness is at least $2^{\min(\mu)}$.
The main result of \cite{DavDymJan10} implies the vcd of a lattice in the building is the same as that of the underlying Coxeter group.
Therefore, $vcd(W_{{\mathrm{Th}^\mu(X)}})=vcd(W_{\mathrm{Th}^1(X)})$, independent of $\mu$.
Since LMSSW computed $vcd(W_{{\mathrm{Th}^{\mu_n}(X_n)}})=n+1$, we conclude $vcd(W_{{\mathrm{Th}^\mu(X_n)}})=vcd(W_{{\mathrm{Th}^1(X_n)}})=vcd(W_{{\mathrm{Th}^{\mu_n}(X_n)}})=n+1$, independent of $\mu$.

LMSSW deduce the existence of virtual algebraic fiberings by choosing $\mu_n$ carefully at each step and playing the  Jankiewicz-Norin-Wise game \cite{JanNorWis21}.
We want to vary $\mu$, so this likely will not work for us. 
Instead, we use a theorem of Kielak combined with a computation of the first $L^2$-Betti number of graph products of large finite groups.

Davis--Dymara--Januszkiewicz--Okun \cite{ddjo} compute the $L^2$-cohomology of a lattice $\Gamma$ in a building $\Phi$ in terms of a weighted $L^2$-cohomology theory of the underlying Coxeter group.
For large weights (depending on the radius of convergence of the growth series of $W$), the weighted $L^2$-cohomology of the Davis complex $\Sigma$ can essentially be identified with the cohomology of compact support $H^*_c(\Sigma)$,  hence with the group cohomology $H^\ast(W,\mathbb{R} W)$, see \cite[Theorem 20.7.6]{davis}.
In particular, if the thickness of the building is sufficiently large, and $H^m(W, \mathbb{R} W) = 0$, then the $m^{th}$ $L^2$-cohomology of the lattice vanishes. 

Davis computed $H^*(W, \mathbb{R} W)$ for arbitrary finitely generated Coxeter groups:
$$H^*(W, \mathbb{R} W) = \bigoplus_{T \in \mathcal{S}} \mathbb{Z} W_T \otimes H^{i-1}(L - \sigma_T, \mathbb{R})$$
where $\sigma_T$ is the simplex of the link $L$ of $W$ corresponding to the spherical subset $T$. 
In particular, $H^1(W, \mathbb{R} W) = 0$ if and only if $L$ is not separated by a simplex (which is equivalent to $W$ being one-ended).
Therefore, $b_1^{(2)}(W) = 0$ if the building is sufficiently thick and $W$ is one-ended. In \cite[Corollary 4.4]{LMSSW-fibering}, it is shown that being one-ended is preserved under the Osajda procedure and thickening, so we can assume this as well.
Now, since $W$ is virtually RFRS, we can use the following theorem of Kielak to show these groups virtually algebraically fiber for large enough thickness:

\begin{reftheorem*}[{Kielak's fibering theorem \cite[Theorem~5.3]{Kie20}}]
  Suppose that $G$ is a virtually RFRS group and $b_1^{(2)}(G) = 0$.
  Then $G$ has a finite-index subgroup $H$ that algebraically fibers. 
\end{reftheorem*}

Thus for each $n$, whenever $\min(\mu)$ is sufficiently large,  $W_{\mathrm{Th}^\mu(X_n)}$ virtually algebraically fibers.

It remains to show that for fixed $n$ there are infinitely many quasi-isometry types of groups 
$W_{\mathrm{Th}^\mu(X_n)}$ with $\min(\mu)$ bounded below. 
A well-known result in graph theory \cite{dirac} implies that a non-complete graph with no separating clique contains a full cycle of length $\ell\geq 4$.
Since $X_n$ is 5--large and defines a one-ended Coxeter group, its 1--skeleton contains a full cycle of length $\ell\geq 5$.
Therefore, $W_{\mathrm{Th}^N(X_n)}$ contains a special subgroup that is a graph product defined on a cycle of length $\ell\geq 5$ with vertex groups $(\mathbb{Z}/2)^N$.
Such a group is a lattice in a Bourdon building \cite{bourdon,caprace}, and for fixed $\ell$ the conformal dimension of the boundary of the building goes to infinity with $N$.
Since special subgroups are convex, this gives a lower bound on the conformal dimension of the boundary of 
$W_{\mathrm{Th}^N(X_n)}$ that increases to~$\infty$ with $N$.
Thus, for fixed $n$ and variable $N$, there are infinitely many quasi-isometry classes of group $W_{\mathrm{Th}^N(X_n)}$, all of which have virtual cohomological dimension $n+1$, and all of which virtually algebraically fiber once $N$ is large enough. 
    \end{proof}

\section{Building round trees in a Davis complex}\label{sec:cdim_lower_bounds}

The idea of a round tree is to take a rooted regular tree $\mathbb T$ with basepoint $x_0$ and rotate the tree about $x_0$ to construct a space that is topologically $\bigl(\mathbb T\times\mathbb{S}^1\bigr)/\bigl(x_0\times \mathbb{S}^1\bigr)$. 
Metrically, one fixes a parameter $\kappa<0$ and asks that for each geodesic ray $\alpha$ in $\mathbb T$ based at $x_0$, the ``sheet'' $\bigl(\alpha \times\mathbb{S}^1\bigr)/\bigl(x_0\times \mathbb{S}^1\bigr)$ is isometric to the plane of constant curvature $\kappa$. 
If a round tree quasi-isometrically embeds in a hyperbolic space, then lower bounds for the conformal dimension of the boundary of the space follow from data of the round tree. 
Actually, similar estimates can be done with only a sector of a round tree, and the regularity and non-positive curvature conditions can be combinatorialized. We follow Mackay \cite{mackay,mackay16} in describing such an object as a ``combinatorial round tree.''

    \subsection{Combinatorial round trees}

    \begin{defn}[Combinatorial round tree] \cite[Definition 7.1]{mackay16} \label{defn:RT}  
        A polygonal $2$-complex $A$ is a \emph{combinatorial round tree} with vertical branching $V \in \N$ and horizontal branching at most $H \in \N$ if the following holds. Let $T = \{1,2, \ldots, V\}$. The space $A$ can be described either as an increasing nested union of finite sub-(round tree) subcomplexes $A_n$, the \emph{round tree at step $n$},  or as a union of \emph{sheets} $F_{\bt}$: 
            \[ A = \bigcup_{n\in\N\cup\{0\}}A_n = \bigcup_{\bt \in T^{\N}} F_{\bt} \]
    In addition, the following hold:
    \begin{enumerate}[1. ]
      \item $A_0$ is homeomorphic to a disk. 
            \item $A$ has a basepoint $x_0$ contained in the boundary of $A_0$ and not contained in the closure of $A_n\setminus A_0$ for any $n\in \N$. 
            \item Each sheet $F_{\bt}$ is an infinite planar $2$-complex homeomorphic to a half-plane whose boundary is the union of two rays $L_{\bt}$ and $R_{\bt}$ with $L_{\bt} \cap R_{\bt} = \{x_0\}$.            
            \item For each $n\in\N$, the round tree at step $n$ satisfies $A_n = \St(A_{n-1})$. Given $\bt = (t_1, t_2, \ldots) \in T^\N$, let $\bt_n = (t_1, \ldots, t_n) \in T^n$.
              If $\bt,\bt' \in T^\N$ satisfy $\bt_n = \bt'_n$ and $\bt_{n+1} \neq \bt'_{n+1}$, then: 
                \[ A_n \cap F_{\bt} \subseteq F_{\bt'} \cap F_{\bt} \subseteq A_{n+1} \cap F_{\bt}\] \label{item:separation_of_sheets}
            \item Each $2$-cell $P \subseteq A_n$ meets at most $VH$ $2$-cells in $A_{n+1} \setminus A_n$.    
        \end{enumerate}
    \end{defn}
    Compared to the round tree, the set $T^{\N}$ plays the role of the tree. The sheet $F_{\bt}$ corresponding to $\bt\in T^{\N}$ will, in practice, be quasi-isometric to a sector in the hyperbolic plane.
    The $A_n$ are concentric balls about $A_0$ in the face metric.
    For $\bt$ and $\bt'$ that agree up to coordinate $n$ and differ at
    coordinate $n+1$, the sheets $F_{\bt}$ and $F_{\bt'}$ agree in
    $A_n$ and diverge by the time they leave $A_{n+1}$. Our \Cref{const:round_tree} will not use the full flexibility allowed by \eqref{item:separation_of_sheets}; we will have that $A_{n+1}\setminus A_n$ is a union of distinct width--1 strips, so $F_{\bt}$ and $F_{\bt'}$ diverge immediately upon leaving $A_n$.
    
     If
$\bt=(t_1,\dots,t_i)\in T^i$ and $t_{i+1}\in T$ we write $(\bt,t_{i+1})$ to mean
$(t_1,\dots,t_i,t_{i+1})\in T^{i+1}$.

A combinatorial round tree can be described inductively, as explained by Mackay \cite{mackay16}.
    For each $n \in \N$ the complex $A_n=\bigcup_{\bt\in T^n}D_{\bt}$,
    where $D_{\bt}:=A_n\cap F_{\bt}$ and  $F_{\bt}$ is the common
    intersection of all $F_{\bar{\bt}}$ such that $\bt$ is a prefix of
    $\bar{\bt}\in T^{\N}$.
    Each $D_{\bt}$ is a planar 2--complex homeomorphic to a disk, and these are glued together along their initial subcomplexes in a branching fashion.
    The boundary of $D_{\bt}$ decomposes as a union of three paths
    $L_{\bt}$, $R_{\bt}$ and $E_{\bt}$, where for any $\bar{\bt}\in
    T^{\N}$ with prefix $\bt$ we have $L_{\bt}:=L_{\bar{\bt}}\cap A_n$ and $R_{\bt}:=R_{\bar{\bt}}\cap A_n$. 
    The complex $A_{n+1}$ is built from $A_n$ by taking each $\bt$ of
    length $n$ and  attaching $2$--cells in $V$--many strips
    along the outer boundary path $E_{\bt}$ of $D_{\bt}$.
    The union of $D_{\bt}$ with each one of these strips form the new disks $D_{(\bt,t')}$, for $(\bt,t')\in T^{n+1}$. 
    Each $2$--cell in $D_{\bt}$ meets at most $H$--many new $2$--cells in each $D_{(\bt,t')}$.

    \begin{thm} \cite[Theorem 7.2]{mackay16} \label{thm:mackayRT}
        Let $X$ be a hyperbolic polygonal $2$-complex, and let $A$ be
        a combinatorial round tree with vertical branching $V \geq 2$
        and horizontal branching $H \geq 2$. Suppose that $A^{(1)}$,
        with the natural length metric giving each edge length one, quasi-isometrically embeds into $X$. Then:
            \[ \Confdim(\p X) \geq 1 + \frac{\log V}{\log H}\]
    \end{thm}

    \subsection{Conditions on the defining graph}

    We specify a local condition, \emph{$(n,m)$--branching}, on the presentation graph~$\Gamma$ of
    a right-angled Coxeter group $W_\Gamma$, which  ensures that a
    combinatorial round tree, with vertical and horizontal
    branching parameters depending on $n$ and $m$, respectively, can be constructed in the Davis complex for $W_{\Gamma}$. The condition guarantees that the links in the Davis complex contain enough $m$--cycles with small overlap so that pieces of a quasi-isometrically embedded hyperbolic plane can be constructed from squares in the complex. See \Cref{figure:n6branching} and \Cref{figure:round_tree}. 

    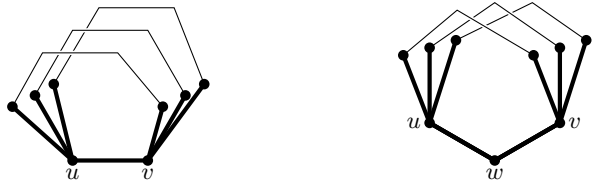
\begin{figure}
      \centering
       \begin{tikzpicture}\small
            \coordinate[label={[label distance=0pt] -90:$$}] (ua) at (-120:1);
            \coordinate[label={[label distance=0pt] -90:$$}] (va) at (-60:1);
            \coordinate[label={[label distance=0pt] 180:$$}] (uia) at ($(180:1)+(.25,.15)$);
            \coordinate[label={[label distance=0pt] 0:$$}] (vja) at ($(0:1)+(.25,.15)$);
            \coordinate[label={[label distance=0pt] 0:$$}] (ijka) at ($(60:1)+(.35,.3)$);
            \coordinate[label={[label distance=0pt] 180:$$}] (ka) at ($(120:1)+(.35,.3)$);
            \draw (vja)--(ijka)--(ka)--(uia);
            \draw[ultra thick] (uia)--(ua) (va)--(vja);
            \fill  (uia) circle (2pt) (vja) circle (2pt);

            \coordinate[label={[label distance=0pt] -90:$u$}] (u) at (-120:1);
            \coordinate[label={[label distance=0pt] -90:$v$}] (v) at (-60:1);
            \coordinate[label={[label distance=0pt] 180:$$}] (ui) at (180:1);
            \coordinate[label={[label distance=0pt] 0:$$}] (vj) at (0:1);
            \coordinate[label={[label distance=0pt] 0:$$}] (ijk) at (60:1);
            \coordinate[label={[label distance=0pt] 180:$$}] (k) at (120:1);
            \draw[ultra thick,white] (vj)--(ijk)--(k)--(ui);
            \draw (vj)--(ijk)--(k)--(ui);
            \draw[ultra thick] (ui)--(u)--(v)--(vj);
            \fill (u) circle (2pt) (v) circle (2pt) (ui) circle (2pt)
            (vj) circle (2pt);

             \coordinate[label={[label distance=0pt] -90:$$}] (ub) at (-120:1);
            \coordinate[label={[label distance=0pt] -90:$$}] (vb) at (-60:1);
            \coordinate[label={[label distance=0pt] 180:$$}] (uib) at ($(180:1)+(-.3,-.15)$);
            \coordinate[label={[label distance=0pt] 0:$$}] (vjb) at ($(0:1)+(-.3,-.15)$);
            \coordinate[label={[label distance=0pt] 0:$$}] (ijkb) at ($(60:1)+(-.4,-.3)$);
            \coordinate[label={[label distance=0pt] 180:$$}] (kb) at ($(120:1)+(-.4,-.3)$);
            \draw[ultra thick, white] (vjb)--(ijkb)--(kb)--(uib);
            \draw (vjb)--(ijkb)--(kb)--(uib);
            \draw[ultra thick] (uib)--(ub) (vb)--(vjb);
            \fill  (uib) circle (2pt) (vjb) circle (2pt);
          \end{tikzpicture}
          \hspace{2cm}
          \begin{tikzpicture}\small
            \coordinate[label={[label distance=-1pt] 180:$$}] (ua) at (-150:1);
            \coordinate[label={[label distance=-1pt] -90:$$}] (wa) at (-90:1);
            \coordinate[label={[label distance=-1pt] 0:$$}] (va) at (-30:1);
            \coordinate[label={[label distance=-1pt] 0:$$}] (vja) at ($(30:1)+(.35,.1)$);
            \coordinate[label={[label distance=-1pt] 90:$$}] (ijka) at ($(90:1)+(.5,.1)$);
            \coordinate[label={[label distance=-1pt] 180:$$}] (uia) at ($(150:1)+(.35,.1)$);
            \draw (uia)--(ijka)--(vja);
            \draw [ultra thick] (uia)--(ua)--(wa)--(va)--(vja);
            \fill (uia) circle (2pt) (vja) circle (2pt);
            \coordinate[label={[label distance=0pt] 180:$u$}] (u) at (-150:1);
            \coordinate[label={[label distance=0pt] -90:$w$}] (w) at (-90:1);
            \coordinate[label={[label distance=0pt] 0:$v$}] (v) at (-30:1);
            \coordinate[label={[label distance=-1pt] 0:$$}] (vj) at (30:1);
            \coordinate[label={[label distance=-1pt] 90:$$}] (ijk) at ($(90:1)+(0,.1)$);
            \coordinate[label={[label distance=-1pt] 180:$$}] (ui) at (150:1);
            \draw[ultra thick, white] (ui)--(ijk)--(vj);
            \draw (ui)--(ijk)--(vj);
            \draw [ultra thick] (ui)--(u)--(w)--(v)--(vj);
            \fill (ui) circle (2pt) (u) circle (2pt) (w) circle (2pt)
            (v) circle (2pt) (vj) circle (2pt);
            \coordinate[label={[label distance=-1pt] 180:$$}] (ub) at (-150:1);
            \coordinate[label={[label distance=-1pt] -90:$$}] (wb) at (-90:1);
            \coordinate[label={[label distance=-1pt] 0:$$}] (vb) at (-30:1);
            \coordinate[label={[label distance=-1pt] 0:$$}] (vjb) at ($(30:1)+(-.35,-.1)$);
            \coordinate[label={[label distance=-1pt] 90:$$}] (ijkb) at ($(90:1)+(-.5,0)$);
            \coordinate[label={[label distance=-1pt] 180:$$}] (uib) at ($(150:1)+(-.35,-.1)$);
            \draw[ultra thick, white] (uib)--(ijkb)--(vjb);
            \draw (uib)--(ijkb)--(vjb);
            \draw [ultra thick] (uib)--(ub)--(wb)--(vb)--(vjb);
            \fill (uib) circle (2pt) (vjb) circle (2pt);
          \end{tikzpicture}
	\caption{An illustration of the $(3,6)$--branching condition. The condition ensures that certain subgraphs, drawn with thickened edges, can be extended to contain induced 6-cycles that intersect only in the original thickened subgraph.}
	\label{figure:n6branching}
\figurealttext{Two diagrams depicting the requirements of the
  $(3,6)$-branching condition. The first shows three hexagonal 2-cells
  sharing a single edge. The second shows three hexagonal 2-cells sharing two adjacent edges.} 
    \end{figure}   

    \begin{defn} \label{defn:n6_branching}
For $n\geq 1$ and $m\geq 5$,  a finite simplicial graph $\G$ has \emph{$(n,m)$--branching} if every vertex has valence at least $n+1$ and
for every induced path $P$ of length 1 or 2 with endpoints $u$ and
$v$, the following conditions are satisfied.
         For every set
          $\{u_1,\dots,u_n\}$ of distinct vertices in
          $\Lk(u)\setminus P$ there exists a set
          $\{v_1,\dots,v_n\}$ of distinct vertices in
          $\Lk(v)\setminus P$ and cycle subgraphs $H_i$ such that
          for all $i\neq j$:
          \begin{itemize}
          \item $5\leq |H_i|\leq m$
          \item $H_i$ contains $\{u_i,v_i\}\cup P$.
          \item $H_i\cap H_j=P$
            \item $\bigcup_iH_i$ is an induced subgraph of $\Gamma$. 
          \end{itemize}
        \end{defn}

        We first derive some basic consequences of this definition and
        use them to show that a right-angled Coxeter group whose presentation graph has
        $(n,m)$--branching is hyperbolic and to deduce the topological
        type of its boundary.
        Then we estimate the conformal 
        dimension of the boundary. 
        Note that $(n,m)$--branching implies $(n',m')$--branching for any $n'\leq n$ and $m'\geq m$. 

      \begin{lemma}\label{lem:branching_implies_high_girth}
       A graph with  $(1,m)$--branching has girth at least 5.
      \end{lemma}
      \begin{proof}
        Suppose $u$, $v$, and $w$ are vertices of a triangle. Take
        $u_1:=w$, which is a neighbor of $u$ distinct from $v$.
        Since $u_1$ is adjacent to $v$, the triangle itself is the only induced cycle containing the
        segment $u_1\edge u\edge v$, but the branching condition requires there
        is such a cycle of length at least 5. Thus, the girth is greater than 3. 

        Suppose $u$, $w$, $v$, $x$ are sequential vertices of a
        square.
        Since the girth is greater than 3, $u\edge w\edge v$ is an
        induced segment of length 2.
        Take $u_1:=x$. Since $u_1$ is adjacent to $v$, the square
        itself is the only induced cycle containing $u_1\edge u\edge w\edge v$, but
        the branching condition requires such a cycle of length at
        least 5. Thus, the girth is greater than 4. 
      \end{proof}

    A triangle-free graph is \emph{inseparable} if it is connected, has no separating complete subgraph, no cut pair, and no separating complete subgraph suspension.

      \begin{lemma}\label{lem:branching_implies_inseparable}
        A connected graph with $(2,m)$--branching is inseparable.
      \end{lemma}
      \begin{proof}
        Suppose $\Gamma$ is a connected graph with
        $(2,m)$--branching. Then, $\Gamma$ is triangle-free by
        \Cref{lem:branching_implies_high_girth}. Thus, it suffices to verify that $\Gamma$ has no cut points, cut edges, cut pairs, or cut paths of length two.
        
        Since $\Gamma$ is not complete it has cut sets. 
        Since it is connected, if $w$ is a vertex belonging to a minimal cut
       set $C$, then every component of $\Gamma\setminus C$
       contains a neighbor of $w$, because otherwise we could remove
       $w$ from $C$ and get a smaller cut set.
       For any such choice of $C$ and $w$, let $u$ and $v$ be neighbors of $w$ that are in different components of $\Gamma\setminus C$.
        The $(2,m)$--branching condition implies there are two cycles
       $\gamma$ and $\gamma'$ whose intersection is the segment
       $u\edge w\edge v$ and whose union is an induced subgraph.
       As $C$ separates $u$ and $v$, it must contain 
       at least $w$, a vertex $x$ from $\gamma\setminus\{u,v,w\}$, and a vertex $y$ from $\gamma'\setminus\{u,v,w\}$.
       Since $\gamma\cup\gamma'$ is induced, $\{w,x,y\}$ is an anticlique.
       Thus, any minimal cut set of $\Gamma$ contains a 3--anticlique.
       
       Let $\mathcal{C}$ be the collection of full subgraphs of $\Gamma$ consisting of single vertices, pairs of nonadjacent vertices, single edges, and 2--paths.
       Inseparability is equivalent to $\mathcal{C}$ not containing a cut set. 
       If $\mathcal{C}$ contains a cut set then it contains a minimal cut set, since $\mathcal{C}$ is closed under passing to nonempty full subgraphs.
       But $\mathcal{C}$ cannot contain a minimal cut set, since none of the sets in $\mathcal{C}$ contain a 3--anticlique as a full subgraph.
     \end{proof}

     \begin{lemma}\label{lem:branching_implies_nonplanar}
      The minimum genus of an orientable surface into which a graph
      with $(n,m)$--branching embeds is bounded below by a superlinear
      function of $n$ that is positive for $n\geq 3$.
      In particular, a graph with $(3,m)$--branching is nonplanar.
      \end{lemma}
      The 1--skeleton of the dodecahedron is an example of a planar graph with
      $(2,9)$--branching. 
      \begin{proof}
        Suppose $\Gamma$ has $(n,m)$--branching and is embedded into an
        orientable surface $\Sigma$ of minimal genus~$g$.
        Minimality implies complementary components of
        $\Sigma\setminus\Gamma$ are homeomorphic to disks, so if we
        call the number of such regions $F$ and let $E$ and $V$ be the
        number of edges and vertices of $\Gamma$, respectively, then
        $2-2g=V-E+F$.
        
        By \Cref{lem:branching_implies_high_girth}, $\Gamma$ has girth
        at least 5, so $5F\leq 2E$.
        The definition of $(n,m)$--branching requires $\Gamma$ to have
        valence at least $n+1$, so $(n+1)V\leq 2E$.
        Therefore:
        \[g\geq 1+\frac{3n-7}{10(n+1)}E\]
        This is enough to conclude nonplanarity when $n\geq 3$.
        To see $g\in\Omega(n^3)$, observe that since
        $\Gamma$ has valence at least $n+1$ and girth at least 5, it
        has at least 
        $(n+1)+(n+1)n+\frac{(n+1)n^2}{2}$ edges.
      \end{proof}

      \begin{prop}\label{prop:branching_usually_implies_menger}
        If $\Gamma$ is a graph with
        $(1,m)$--branching, then $W_\Gamma$ is hyperbolic.
        If $\Gamma$ is connected and has $(2,m)$--branching then $\bdry
        W_\Gamma$ is either the Sierpinski carpet or the Menger curve,
        according to whether $\Gamma$ is planar or not.
        If $\Gamma$ is connected with $(3,m)$--branching, then $\bdry W_\Gamma$ is
        the Menger curve.
      \end{prop}
      \begin{proof}
        The graph $\Gamma$ has girth
        at least 5 by \Cref{lem:branching_implies_high_girth}, so $\Gamma$ is triangle-free and $W_\Gamma$ is
        hyperbolic by \Cref{lemma:RACG_props}.
        If $\Gamma$ is connected with $(2,m)$--branching then $\Gamma$ is inseparable by \Cref{lem:branching_implies_inseparable}.
        Thus, planarity of $\Gamma$ determines whether $\bdry
        W_\Gamma$ is the Sierpinski carpet or the Menger curve by
        \cite[Corollary~1.7]{MR4563333} and related works
        \cite{Swi16,MR3934913,DanKapSwi24,MR4453753}.
        The boundary is homeomorphic to the Menger curve if $\Gamma$ is connected and has $(3,m)$--branching, since 
        \Cref{lem:branching_implies_nonplanar} forces nonplanarity.
      \end{proof}

Now we turn to the main objective of this section:
    \begin{thm} \label{thm:confdim_lower_bounds}
        If $\Gamma$ is a graph with $(n,m)$--branching
      then $W_\Gamma$ is hyperbolic and satisfies:
            \[ \Confdim(\p W_\Gamma) \geq 1 + \frac{\log n}{\log 3m-7}\]
          \end{thm}

          \Cref{thm:confdim_lower_bounds} is established by applying
          \Cref{thm:mackayRT} to the output of \Cref{const:round_tree}
          below.
          This construction is inspired by work of
    Mackay~\cite{mackay}; see also \cite{fieldguptalymanstark}.

    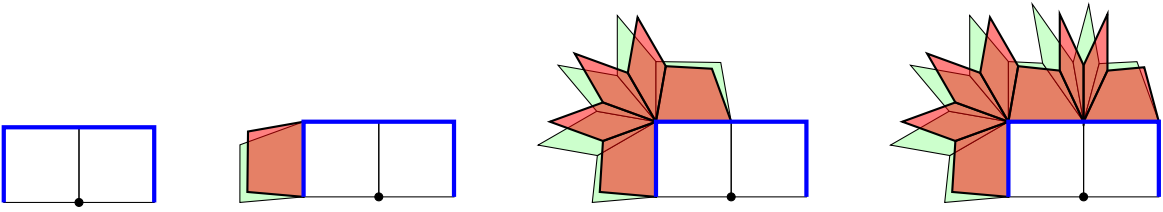
\begin{figure}
      \centering
     \begin{tikzpicture}
       \coordinate[label={[label distance=-1pt] 0:$$}] (o) at (0,0);
       \coordinate[label={[label distance=-1pt] 0:$$}] (w) at (180:1);
       \coordinate[label={[label distance=-1pt] 0:$$}] (x) at (90:1);
       \coordinate[label={[label distance=-1pt] 0:$$}] (y) at (0:1);
       \coordinate[label={[label distance=-1pt] 0:$$}] (wx) at ($(w)+(x)$);
       \coordinate[label={[label distance=-1pt] 0:$$}] (xw) at (wx);
       \coordinate[label={[label distance=-1pt] 0:$$}] (xy) at ($(x)+(y)$);
       \coordinate[label={[label distance=-1pt] 0:$$}] (yx) at (xy);
       \draw (o)--(w)--(xw)--(x)--(o);
       \draw (o)--(x)--(xy)--(y)--(o);

       \draw[ultra thick,blue] (w)--(wx)--(x)--(xy)--(y);
       \filldraw (o) circle (1.5pt);
    \end{tikzpicture}
    \hfill
    \begin{tikzpicture}
       \coordinate[label={[label distance=-1pt] 0:$$}] (o) at (0,0);
       \coordinate[label={[label distance=-1pt] 0:$$}] (w) at (180:1);
       \coordinate[label={[label distance=-1pt] 0:$$}] (x) at (90:1);
       \coordinate[label={[label distance=-1pt] 0:$$}] (y) at (0:1);
       \coordinate[label={[label distance=-1pt] 0:$$}] (wx) at ($(w)+(x)$);
       \coordinate[label={[label distance=-1pt] 0:$$}] (xw) at (wx);
       \coordinate[label={[label distance=-1pt] 0:$$}] (xy) at ($(x)+(y)$);
       \coordinate[label={[label distance=-1pt] 0:$$}] (yx) at (xy);
       \draw (o)--(w)--(xw)--(x)--(o);
       \draw (o)--(x)--(xy)--(y)--(o);

       \coordinate[label={[label distance=-1pt] 0:$$}] (pwa) at ($(w)+(185:.85)$);
       \coordinate[label={[label distance=-1pt] 0:$$}] (pwxa) at ($(wx)+(200:.9)$);
     
       \draw[fill=green!20] (wx)--(pwxa)--(pwa)--(w)--cycle;

       \coordinate[label={[label distance=-1pt] 0:$$}] (wa) at ($(w)+(175:.75)$);
       \coordinate[label={[label distance=-1pt] 0:$$}] (wxa) at ($(wx)+(190:.75)$);
      
       \draw[thick, fill=red, fill opacity=0.5] (wx)--(wxa)--(wa)--(w) --cycle;

       \draw[ultra thick,blue] (w)--(wx)--(x)--(xy)--(y);
       \filldraw (o) circle (1.5pt);
    \end{tikzpicture}
    \hfill
    \begin{tikzpicture}
       \coordinate[label={[label distance=-1pt] 0:$$}] (o) at (0,0);
       \coordinate[label={[label distance=-1pt] 0:$$}] (w) at (180:1);
       \coordinate[label={[label distance=-1pt] 0:$$}] (x) at (90:1);
       \coordinate[label={[label distance=-1pt] 0:$$}] (y) at (0:1);
       \coordinate[label={[label distance=-1pt] 0:$$}] (wx) at ($(w)+(x)$);
       \coordinate[label={[label distance=-1pt] 0:$$}] (xw) at (wx);
       \coordinate[label={[label distance=-1pt] 0:$$}] (xy) at ($(x)+(y)$);
       \coordinate[label={[label distance=-1pt] 0:$$}] (yx) at (xy);
       \draw (o)--(w)--(xw)--(x)--(o);
       \draw (o)--(x)--(xy)--(y)--(o);

       \coordinate[label={[label distance=-1pt] 0:$$}] (pwa) at ($(w)+(185:.85)$);
       \coordinate[label={[label distance=-1pt] 0:$$}] (pwxa) at ($(wx)+(210:.9)$);
       \coordinate[label={[label distance=-1pt] 0:$$}] (pwxb) at ($(wx)+(170:.8)$);
       \coordinate[label={[label distance=-1pt] 0:$$}] (pwxc) at ($(wx)+(130:.8)$);
       \coordinate[label={[label distance=-1pt] 0:$$}] (pwxd) at ($(wx)+(90:.8)$);
       \coordinate[label={[label distance=-1pt] 0:$$}] (pxd) at ($(x)+(100:.8)$);
     
       \draw[fill=green!20] (wx)--(pwxa)--(pwa)--(w)--cycle;
       \draw[fill=green!20]  (pwxa)--($(pwxa)+(pwxb)-(wx)$)--(pwxb)--(wx)--cycle;
       \draw[fill=green!20]  (pwxb)--($(pwxb)+(pwxc)-(wx)$)--(pwxc)--(wx)--cycle;
       \draw[fill=green!20]  (pwxc)--($(pwxc)+(pwxd)-(wx)$)--(pwxd)--(wx)--cycle;
       \draw[fill=green!20]  (wx)--(pwxd)--(pxd)--(x)--cycle;
 
       \coordinate[label={[label distance=-1pt] 0:$$}] (wa) at ($(w)+(175:.75)$);
       \coordinate[label={[label distance=-1pt] 0:$$}] (wxa) at ($(wx)+(200:.75)$);
       \coordinate[label={[label distance=-1pt] 0:$$}] (wxb) at ($(wx)+(160:.75)$);
       \coordinate[label={[label distance=-1pt] 0:$$}] (wxc) at ($(wx)+(120:.75)$);
       \coordinate[label={[label distance=-1pt] 0:$$}] (wxd) at ($(wx)+(80:.75)$);
       \coordinate[label={[label distance=-1pt] 0:$$}] (xd) at ($(x)+(110:.75)$);
      
       \draw[thick, fill=red, fill opacity=0.5] (wx)--(wxa)--(wa)--(w) --cycle;
       \draw[thick, fill=red, fill opacity=0.5] (wx)--(wxa)--($(wxa)+(wxb)-(wx)$)--(wxb)--cycle;
       \draw[thick, fill=red, fill opacity=0.5] (wx)--(wxb)--($(wxb)+(wxc)-(wx)$)--(wxc)--cycle;
       \draw[thick, fill=red, fill opacity=0.5] (wx)--(wxc)--($(wxc)+(wxd)-(wx)$)--(wxd)--cycle;
       \draw[thick, fill=red, fill opacity=0.5]
       (wx)--(wxd)--(xd)--(x)--cycle;

       \draw[ultra thick,blue] (w)--(wx)--(x)--(xy)--(y);
       \filldraw (o) circle (1.5pt);
    \end{tikzpicture}
        \hfill
      \begin{tikzpicture}
       \coordinate[label={[label distance=-1pt] 0:$$}] (o) at (0,0);
       \coordinate[label={[label distance=-1pt] 0:$$}] (w) at (180:1);
       \coordinate[label={[label distance=-1pt] 0:$$}] (x) at (90:1);
       \coordinate[label={[label distance=-1pt] 0:$$}] (y) at (0:1);
       \coordinate[label={[label distance=-1pt] 0:$$}] (wx) at ($(w)+(x)$);
       \coordinate[label={[label distance=-1pt] 0:$$}] (xw) at (wx);
       \coordinate[label={[label distance=-1pt] 0:$$}] (xy) at ($(x)+(y)$);
       \coordinate[label={[label distance=-1pt] 0:$$}] (yx) at (xy);
       \draw (o)--(w)--(xw)--(x)--(o);
       \draw (o)--(x)--(xy)--(y)--(o);

       \coordinate[label={[label distance=-1pt] 0:$$}] (pwa) at ($(w)+(185:.85)$);
       \coordinate[label={[label distance=-1pt] 0:$$}] (pwxa) at ($(wx)+(210:.9)$);
       \coordinate[label={[label distance=-1pt] 0:$$}] (pwxb) at ($(wx)+(170:.8)$);
       \coordinate[label={[label distance=-1pt] 0:$$}] (pwxc) at ($(wx)+(130:.8)$);
       \coordinate[label={[label distance=-1pt] 0:$$}] (pwxd) at ($(wx)+(90:.8)$);
       \coordinate[label={[label distance=-1pt] 0:$$}] (pxd) at ($(x)+(125:.95)$);
       \coordinate[label={[label distance=-1pt] 0:$$}] (pxe) at ($(x)+(100:.8)$);
       \coordinate[label={[label distance=-1pt] 0:$$}] (pxf) at ($(x)+(75:.8)$);
       \coordinate[label={[label distance=-1pt] 0:$$}] (pxyf) at ($(xy)+(110:.85)$);
       \draw[fill=green!20] (wx)--(pwxa)--(pwa)--(w)--cycle;
       \draw[fill=green!20]  (pwxa)--($(pwxa)+(pwxb)-(wx)$)--(pwxb)--(wx)--cycle;
       \draw[fill=green!20]  (pwxb)--($(pwxb)+(pwxc)-(wx)$)--(pwxc)--(wx)--cycle;
       \draw[fill=green!20]  (pwxc)--($(pwxc)+(pwxd)-(wx)$)--(pwxd)--(wx)--cycle;
       \draw[fill=green!20]  (wx)--(pwxd)--(pxd)--(x)--cycle;
       \draw[fill=green!20]  (x)--(pxd)--($(pxd)+(pxe)-(x)$)--(pxe)--cycle;
       \draw[fill=green!20]  (x)--(pxe)--($(pxe)+(pxf)-(x)$)--(pxf)--cycle;
       \draw[fill=green!20]  (x)--(pxf)--(pxyf)--(xy)--cycle;

       \coordinate[label={[label distance=-1pt] 0:$$}] (wa) at ($(w)+(175:.75)$);
       \coordinate[label={[label distance=-1pt] 0:$$}] (wxa) at ($(wx)+(200:.75)$);
       \coordinate[label={[label distance=-1pt] 0:$$}] (wxb) at ($(wx)+(160:.75)$);
       \coordinate[label={[label distance=-1pt] 0:$$}] (wxc) at ($(wx)+(120:.75)$);
       \coordinate[label={[label distance=-1pt] 0:$$}] (wxd) at ($(wx)+(80:.75)$);
       \coordinate[label={[label distance=-1pt] 0:$$}] (xd) at ($(x)+(115:.75)$);
       \coordinate[label={[label distance=-1pt] 0:$$}] (xe) at ($(x)+(90:.75)$);
       \coordinate[label={[label distance=-1pt] 0:$$}] (xf) at ($(x)+(65:.75)$);
       \coordinate[label={[label distance=-1pt] 0:$$}] (xyf) at ($(xy)+(105:.75)$);
       \draw[thick, fill=red, fill opacity=0.5] (wx)--(wxa)--(wa)--(w) --cycle;
       \draw[thick, fill=red, fill opacity=0.5] (wx)--(wxa)--($(wxa)+(wxb)-(wx)$)--(wxb)--cycle;
       \draw[thick, fill=red, fill opacity=0.5] (wx)--(wxb)--($(wxb)+(wxc)-(wx)$)--(wxc)--cycle;
       \draw[thick, fill=red, fill opacity=0.5] (wx)--(wxc)--($(wxc)+(wxd)-(wx)$)--(wxd)--cycle;
       \draw[thick, fill=red, fill opacity=0.5] (wx)--(wxd)--(xd)--(x)--cycle;
       \draw[thick, fill=red, fill opacity=0.5] (x)--(xd)--($(xd)+(xe)-(x)$)--(xe)--cycle;
       \draw[thick, fill=red, fill opacity=0.5] (x)--(xe)--($(xe)+(xf)-(x)$)--(xf)--cycle;
       \draw[thick, fill=red, fill opacity=0.5]
       (x)--(xf)--(xyf)--(xy)--cycle;

       \draw[ultra thick,blue] (w)--(wx)--(x)--(xy)--(y);
       \filldraw (o) circle (1.5pt);
    \end{tikzpicture}
    \caption{\small{ The initial steps of the round tree construction with $(2,6)$--branching. The initial subcomplex $A_0$ consists of two squares incident to the marked vertex. The outer edge path is drawn in blue.
    At the first blue edge we attach two new squares, colored red and green.
        At the first interior vertex of the blue path we add four more red squares such that the five red squares plus the square from $A_0$ make a 6--cycle of squares. Similarly add four more green squares to make a second 6--cycle of squares sharing only the $A_0$ square with the red cycle. The $(2,6)$--branching condition makes these choices possible.
        Continue along the blue path, extending the red and green `strips' using the branching condition at each successive interior vertex of the blue path.
}}
\label{figure:round_tree}
\figurealttext{Diagram illustrating the first stages of the round tree construction with $(2,6)$-branching. Two initial squares meet at a marked vertex. A blue outer edge path is extended by attaching red and green strips of squares. At successive vertices of the blue path, the red and green squares are completed to 6-cycles sharing only the prescribed initial square.}
    \end{figure}

    \begin{construction}[Combinatorial round tree in Davis complex] \label{const:round_tree}
      Suppose $\Gamma$ is a graph with $(n,m)$--branching.
      We recursively construct
        a combinatorial round tree $A$ that is a convex subcomplex of
        the Davis complex $X_\Gamma$ of $W_\Gamma$ with
    vertical branching $n$ and horizontal branching at most $3m-7$.

Pick a base vertex $x_0$ in $X_\Gamma$ and an induced edge path $P$ in its link $\Lk(x_0)$.
This corresponds to a chain of squares $A_0 \subset X_\Gamma$ containing $x_0$, such that  $\Lk_{A_0}(x_0)=P$.
The fact that $P$ is induced implies $A_0$ is convex \cite[Lemma~2.11]{haglundwise}.
For illustrative purposes, in \Cref{figure:round_tree} we chose $P$ of
length 2, but any choice between 1 and $\text{girth}(\Gamma)-2$ would
have been possible.

Label one boundary edge of $A_0$ incident to $x_0$ by $L_0$, and label the other boundary edge of $A_0$ incident to $x_0$ by $R_0$.
The outer edge path $E_0$ of $A_0$ is the rest of the boundary; it  is the edge path that is the closure of $\bdry A_0\setminus (L_0\cup R_0)$. 
See \Cref{figure:round_tree}. 

Let $T = \{1, \ldots, n\}$, and let $T^0 = (\, )$, $D_{(\,)}:=A_0$, $L_{(\,)}:=L_0$ and $E_{(\,)}:=E_0$.

Start with $\bt=(\, )$.  We are going to build $n$--many distinct strips attached along $E_{\bt}$.
Suppose $E_{\bt} = y_0\edge \cdots\edge y_{r+1}$ with $y_0 \in L_{\bt}$ and $y_{r+1} \in R_{\bt}$. 
The link  $\Lk_{D_{\bt}}(y_0)$ is a single edge $u\edge v$.  Suppose the edge $y_0\edge y_1$ is labeled $v$, and the edge of $L_{\bt}$ incident to $y_0$ is labeled $u$.  
As the valence of each vertex of $\Gamma$ is at least $n+1$, there exists a set of distinct vertices $\{v_1,\dots,v_n\}\subset\Lk_\Gamma(v)\setminus\{u\}$.  Correspondingly, for each $i\in T$, there is a square in $X_\Gamma$ containing $y_0\edge y_1$ whose sides are labeled $v$ and $v_i$.  
Define these to be the respective first squares of the $n$--many new strips.  
The links of $y_0$ and $y_1$ in the subcomplex obtained by attaching these squares are trees of diameter 2 or 3, so they are induced subgraphs of $\Gamma$, which has girth at least 5.  
Thus, the new  subcomplex remains locally convex,
   hence, convex \cite[Lemma~2.11]{haglundwise}.
For each $i\in T$, let $L_{(\bt,i)}$ be the union of  $L_{\bt}$ and
the edge labeled $v_i$ at $y_0$.

 Now continue along the outer edge path $E_{\bt}$ extending each of the $n$--many strips as depicted in \Cref{figure:round_tree}.
Suppose the construction has progressed to vertex $y_k\in E_{\bt}$.
Let $u$ be the label of the edge $y_{k-1}\edge y_k$, and let $u_i$ be the label such that the last square of strip $i$ is labeled by $u$ and $u_i$. 
If $k=r+1$ we are done extending the strips; define $R_{(\bt,i)}$ to be the union of $R_{\bt}$ and the edge labeled $u_i$ incident to $y_{r+1}$.
Otherwise, extend the strips at $y_k$ in one of two possible ways according to whether $\Lk_{D_{\bt}}(y_k)$ is a single edge or a 2--path.
              
 \begin{enumerate}[1. ]
  \item If $\Lk_{D_{\bt}}(y_k)$ is a single edge $u\edge v$ then $v$ is the
    label of edge $y_k\edge y_{k+1}$ of $E_{\bt}$.
    \Cref{defn:n6_branching} implies that given $\{u_1,\dots,u_n\}$ as
    above there are distinct $\{v_1,\dots,v_n\}\subset
    \Lk_\Gamma(v)\setminus\{u\}$ and cycles $H_i$ intersecting exactly
    in $u\edge v$ whose union is an induced subgraph.
    Cycle $H_i$ implies there is a cycle of squares at $y_k$ that
    contains the square of $D_{\bt}$ containing $y_k$ and the last square
    of strip $i$. 
   The fact that $\cup _i H_i$ is induced in $\Gamma$ implies after
   adding all of these new squares to the $n$--many strips, the subcomplex remains locally convex,
   hence, convex.
  \item If $\Lk_{D_{\bt}}(y_k)$ is an induced segment $u\edge w\edge v$ then
    $y_k$ belongs to two squares of $D_{\bt}$: edge $y_{k-1}\edge y_k$ of
    $E_{\bt}$ is labeled $u$ and is contained in a square of
    $D_{\bt}$ with sides labeled $u$ and $w$, and edge
    $y_k\edge y_{k+1}$ of $E_{\bt}$ is labeled $v$ and is contained in a square of $D_{\bt}$ with sides labeled $w$ and $v$, and sharing a side labeled $w$ with the other square.
    There is an induced path $u\edge w\edge v$ in $\Gamma$. 
    Proceed just as in the previous case using cycles containing $u\edge w\edge v$.
              \end{enumerate}

  In this way we construct $n$ distinct strips
  $S_{(\bt,1)},\dots,S_{(\bt,n)}$, each homeomorphic to a disk, each
  intersecting the previous subcomplex only in $E_{\bt}$, and, by construction, such that
  the new outer boundary $E_{(\bt,i)}$ of strip $S_{(\bt,i)}$ has the
  property that the link of the first vertex is an edge, and the link of each subsequent vertex is either an edge or an
  induced path of  length 2.
Let $D_{(\bt,i)}:=D_{\bt}\cup S_{(\bt,i)}$ 

The description thus far applies to all values of $\bt$.  When $\bt = ()$, let 
$A_1:=A_0\cup\bigcup_{1\leq i\leq n}S_{(\bt,i)}$.
The outer boundary of $A_1$ is $n$--many disjoint arcs $E_{\bt}$ for $\bt\in
T^1$, and $\Lk_{D_{\bt}}(v)$ is an induced path of length one  or two for each $\bt$ and each $v\in E_{\bt}$ (and always of length one for the first vertex of $E_{\bt}$).

Repeat this construction inductively: having constructed $A_j$ whose
outer boundary consists of disjoint arcs $E_{\bt}$ for $\bt\in T^{j-1}$, use the above procedure to build $n$ new strips $S_{(\bt,i)}$ attached along $E_{\bt}$ and take
$D_{(\bt,i)}:=D_{\bt}\cup S_{(\bt,i)}$ and $A_{j+1}:=A_j\cup\bigcup_{1\leq
  i\leq n}S_{(\bt,i)}$.

We have constructed $A=\bigcup  A_j$ that is a combinatorial round
tree with vertical branching~$n$.
To see the horizontal branching is $3m-7$, consider a square $Q$
in $A_j$ that intersects an outer edge path $E_{\bt}$.
By construction, $Q\cap E_{\bt}$ is a segment of length either 1 or 2.
If $Q\cap E_{\bt}$ is a single edge then each of its vertices has link
in $D_{\bt}$ either a single edge, if it is on the boundary of
$E_{\bt}$, or 2 edges, otherwise, so it meets either 1 or at most
$m-2$ squares from each new strip, respectively.
The two vertices of $Q\cap E_{\bt}$ have one square in each strip that
is a common acquaintance, so $Q$ meets at most $2m-5$ squares from each new strip. 
If $Q\cap E_{\bt}$ has length 2, the middle vertex of this segment
meets at most $m-1$ squares from each new strip and the others meet at
most $m-2$.
There are two squares that were double counted, so $Q$ meets at most
$3m-7$ new squares from each new strip.
    \end{construction}

\section{Pontryagin sphere boundaries}\label{sec:Pontryagin}

\begin{thm}\label{thm:Psphere_lower_bound}
There exists a family of hyperbolic right-angled Coxeter groups $W_i$ with Pontryagin sphere boundary and with conformal dimension tending to infinity as $i \rightarrow \infty$. 
\end{thm}

To prove the theorem we use a relative version of a construction of Dranishnikov~\cite{MR1684267}.

\begin{lemma}\label{lem:general_relative_Dranishnikov}
Suppose that $L$ is a flag-no-square subcomplex of a 2-dimensional
simplicial complex $X$.
Then there is a subdivision $X'$ of $X$ that is flag-no-square and contains $L$ as a full subcomplex.
\end{lemma}
\begin{proof}
By taking a partial barycentric subdivision (adding barycenters to edges and faces in $X - L$), we can assume that $X$ is flag and $L$ is a full subcomplex. 
Since $L$ is full in $X$, every $2$-simplex $\sigma$ of $X$ that is
not entirely contained in $L$ intersects $L$ in at most one edge.
Form $X'$ by replacing each such $\sigma$ with one of the simplices in
\Cref{f:dranishnikov}; use the absolute replacement if $\sigma$ does
not contain an edge of $L$, and use relative replacement if it does, where the bottom (red) edge of relative replacement is identified with the edge of $L$.
Since no edges were added with both endpoints in $L$, the subcomplex $L$ remains full in $X'$.
\begin{figure}[h]
  \centering
  \makebox[.5\textwidth][c]{%
  \begin{minipage}[b]{0.4\textwidth}
    \centering
    \begin{tikzpicture}[scale=1.35]
      \begin{scope}[xshift = -1.5cm]
        \coordinate (A) at (0,0);
        \coordinate (B) at (2,0);
        \coordinate (C) at (1,1.732); 

        \draw[thick] (A) -- (B) -- (C) -- cycle;
        \draw[thick] (.5, .8516) -- (1.5, .8516) -- (1,0) -- cycle;
        \draw[thick] (.75, .4258) -- (1.25, .4258) -- (1,.8516) -- cycle;
        \coordinate (G) at (barycentric cs:A=.5,B=1.5,C=1.2);
        \coordinate (H) at (barycentric cs:A=1,B=1,C=.5);
        \coordinate (I) at (barycentric cs:A=1.5,B=.5,C=1.2);
        \draw[thick] (A) -- (.75, .4258);
        \draw[thick] (B) -- (1.25, .4258);
        \draw[thick] (C) -- (1, .8516);
      \end{scope}  
    \end{tikzpicture}\\[1ex]
     (a) Absolute
   \end{minipage}
   \hfill
  \begin{minipage}[b]{0.4\textwidth}
    \centering
    \begin{tikzpicture}[scale=1.35]
      \coordinate (x) at (90:1.15);
      \coordinate (y) at (-30:1.15);
      \coordinate (z) at (210:1.15);
      \coordinate (xy) at ($(x)!.5!(y)$);
      \coordinate (yz) at ($(y)!.5!(z)+(0,.25)$);
      \coordinate (zx) at ($(x)!.5!(z)$);
      \coordinate (a) at ($(zx)!.5!(xy)$);
      \coordinate (b) at ($(xy)!.5!(yz)$);
      \coordinate (c) at ($(yz)!.5!(zx)$);
      \coordinate (ab) at ($(a)!.5!(b)$);
      \coordinate (bc) at ($(b)!.5!(c)$);
      \coordinate (ca) at ($(c)!.5!(a)$);
      \draw[thick] (x)--(xy)--(y)--(yz)--(z)--(zx)--(x) (x)--(a) (y)--(b) (z)--(c) (a)--(xy)--(b)--(yz)--(c)--(zx)--(a) (zx)--(ca) (xy)--(ab) (yz)--(bc) (a)--(ab)--(b)--(bc)--(c)--(ca)--(a) (ab)--(bc)--(ca)--(ab);
      \draw[thick, red] (y)--(z);
    \end{tikzpicture}\\[1ex]
    (b) Relative
  \end{minipage}
  }
  \caption{The replacement simplices for the absolute and relative
    subdivisions.}\label{f:dranishnikov}
  \figurealttext{Two diagrams showing replacement subdivisions of a triangle. Panel A shows the absolute replacement, in which the whole triangle is subdivided into smaller triangles. Panel B shows the relative replacement, in which the triangle is subdivided while a distinguished boundary edge, drawn in red, is preserved.}
\end{figure}
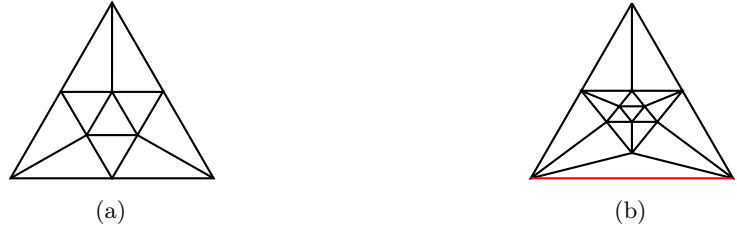

      Suppose $C$ is an induced $4$-cycle in $X'$. Then $C$ is not contained in $L$ since $L$ is flag-no-square. 
      Suppose that $C$ contains a new vertex in the interior of a $2$-simplex $\sigma$ of $X$. Then at least $3$ of the vertices of $C$ are contained in $\sigma$. If all vertices of $C$ are in $\sigma$ we are done since each replacement $2$-simplex is flag-no-square and full in $X'$. 
      Otherwise, by inspection of the picture and since $C$ is induced, $C$ would have to contain two barycenters of edges of $X$. Since these are contained in a unique $2$-simplex of $X$, this is impossible.
      
    Lastly, suppose that $C$ is not contained in $L$ and has no vertices in the interior of a 2-simplex of $X$. 
    Thus, $C$ contains a barycenter of an edge of $X$. Therefore, $C$ viewed as a path in $X$ must be a $3$-cycle. Since $X$ is flag it would span a simplex, which is not contained in $L$ as the edge is subdivided. Since $L$ is full, at least one of the other $2$ edges is not contained in $L$, and hence is subdivided in $X'$. Thus, $C$ would not be a 4-cycle, a contradiction. 
    \end{proof}

\begin{lemma}\label{lem:rel_Dranishnikov_for_graph}
Suppose that $\Gamma$ is a finite, simplicial graph of girth at least 5.
Then there is a flag-no-square triangulation of an orientable closed surface containing $\Gamma$ as an induced subgraph. 
\end{lemma}

\begin{proof}
Since $\Gamma$ has girth at least $5$, it is flag-no-square as a $1$-dimensional simplicial complex.
By \Cref{lem:general_relative_Dranishnikov}, it therefore suffices to show that $\Gamma$ can be embedded as a subcomplex of some triangulation of a closed orientable surface.
That any graph embeds into an orientable surface is well-known, for example one can start with a general position map of $\Gamma \rightarrow S^2$ and then resolve intersections between edges by gluing in handles. 

That $\Gamma$ can be embedded as a subcomplex of some triangulation is also easy, and follows for example from \cite[Corollary 3, p.468]{akin}, so we will only sketch an argument. If the surface $\Sigma$ is chosen to have minimal genus, then each complementary component of $\Sigma-\Gamma$ is homeomorphic to a $2$-disk, so we can take $\Gamma_i$ to be the 1--skeleton of a finite 2--dimensional CW--complex structure on $\Sigma_i$. Triangulating each $2$-cell finely enough guarantees that the resulting cell complex structure on $\Sigma_i$ is simplicial. 
\end{proof}

\begin{proof}[Proof of \Cref{thm:Psphere_lower_bound}]
There exist, as in \Cref{sec:branching_from_projective_planes},
finite, connected, simplicial graphs $\Gamma_i$ with $(n_i,m)$--branching for $m$ fixed and  $3\leq n_i\to \infty$.
\Cref{lem:branching_implies_high_girth} and \Cref{lem:branching_implies_nonplanar} imply $\Gamma_i$ is nonplanar with girth at least 5.
Let $\Sigma_i$ be a flag-no-square triangulated surface containing $\Gamma_i$ as an induced subgraph of its 1--skeleton $\Lambda_i$, as in \Cref{lem:rel_Dranishnikov_for_graph}.
Since $\Gamma_i$ is nonplanar, $\Sigma_i$ has positive genus.
    
The right-angled Coxeter group with defining graph $\Lambda_i$ is
hyperbolic with Pontryagin sphere boundary, by a theorem of Fischer
\cite{fischer} (see also \cite[Theorem 2]{swiatkowski-trees}), and $W_{\Gamma_i}$ is a quasi-convex subgroup of $W_{\Lambda_i}$, so $\Confdim(\bdry W_{\Lambda_i}) \geq \Confdim(\bdry W_{\Gamma_i}) \geq 1 + \frac{\log n_i}{\log 3m-7}$, by \Cref{thm:confdim_lower_bounds}.
\end{proof}

\section{Examples of graphs with
  \texorpdfstring{$(n,m)$}{(n,m)}--branching}\label{sec:branching_from_projective_planes}


In this section we exhibit families of graphs satisfying
\Cref{defn:n6_branching}, so \Cref{thm:confdim_lower_bounds} is not
vacuous.
To mesh well with the branching condition, we want graphs with many
girth-cycles, which suggests candidates related to \emph{edge-girth-regular} graphs \cite{JajKisMik18,AraLee22}.
We discuss two abstract families of graphs, \emph{generalized
  $m$--gons} in \Cref{sec:genm}, and Levi graphs of \emph{transversal
  designs} in \Cref{sec:td}.
Within these abstract families there are some concrete, well-known structures, namely the Levi graphs of the finite projective,
affine, and biaffine planes over the finite field $\F_q$.
In these subclasses we use the explicit nature of the constructions to get
better branching estimates; in fact, in all three cases we find that the branching condition is the best possible, in the sense that the graph satisfies $(n,m)$--branching with $n$ one less than the minimum valence, and $m$ equal to the girth. 
The sole exception is the Fano plane, in which $n$ is two less than the minimum valence.

Many of our examples are Levi graphs of finite incidence
structures.
Such a structure comes with a set of \emph{points}, a set of
\emph{lines}/\emph{blocks}, and an incidence relation saying which
points belong to which lines/blocks. 
The \emph{Levi graph} (or incidence graph) of the incidence structure
has a vertex for each point and line/block, and an edge between two
vertices if they are incident.
This makes a bipartite simple graph.
All of the incidence systems we consider have the property that there
is at most one line/block containing a given pair of points, so the
girth of the Levi graph is at least 6.

We formulate an easy lemma that we use when verifying unions of
cycles are induced subgraphs:
\begin{lemma}\label{lem:usually_induced}
  If $\Gamma$ is a simplicial graph of finite girth $g$ and $A$ and $B$
  are $g$--gons, ie, cycle subgraphs of length $g$,
  with consecutive vertices $a_0,a_1$,\dots, $a_{g-1}$ and $b_0$,
$b_1$,\dots,$b_{g-1}$, respectively, such that $a_i=b_i$ for $0\leq
i\leq k$ for some $1\leq k \leq g-2$ and $a_{k+1}\neq
b_{k+1}$ and $a_{g-1}\neq b_{g-1}$, then the intersection of $A$ and
$B$ is exactly the shared segment $P=a_0\edge\cdots\edge a_k$, and $A\cup B$ is
an induced subgraph unless $k=1$ and $a_i$ is adjacent to $b_j$ for
some choice of $i$ and $j$ with $\{i,j\}=\{\lceil\frac{g+1}{2}\rceil,\lfloor\frac{g+1}{2}\rfloor\}$.
\end{lemma}
\begin{proof}
It is immediate from the girth restriction that $A$ and $B$
individually are induced subgraphs. 
Suppose that for some $k<i, j<g$ we have
$d(a_i,b_j)=\epsilon\in\{0,1\}$.
Write $A$ and $B$ as 
concatenations $A=P+A_1+A_2$, where $a_i=A_1\cap A_2$, and
$B=P+B_1+B_2$, where $B_1\cap B_2=b_j$.
Consider the four paths $\bar A_1+B_1$, $A_2+\bar B_2$, $\bar A_1+\bar
P+B_2$, and $A_2+P+\bar B_1$, where overbar denotes the reverse path.
Each of these paths is either closed, if
$\epsilon=0$, or can be closed by adding the single edge from $a_i$ to
$b_j$, if $\epsilon=1$.
This yields four loops, each of which has length at least $g$. So, there are 
four inequalities in terms of $i,j,k,g,\epsilon$ whose only solutions are $k=\epsilon=1$, $i+j=g+1$, and
$|i-j|\leq 1$.
\end{proof}

\subsection{Generalized \texorpdfstring{$m$}{m}--gons}\label{sec:genm}
\begin{defn}
  A \emph{generalized $m$--gon of order $q$}, for $m,q\geq 2$ is a
  bipartite graph of diameter $m$ and girth $2m$ such that every
  vertex has valence at least $q+1$. 
\end{defn}

\begin{prop}\label{prop:branching_for_genm}
  A generalized $m$--gon of order $q$ has $(n,2m)$--branching, where $n=q$ for $m>3$ and $n=\lfloor q+1-\sqrt{q}\rfloor$ for $m=3$.
\end{prop}
The result is optimal for $m>3$, since $q$ is one less than the valence. 
It is well-known that generalized 3--gons satisfy the axioms of projective planes. In \Cref{sec:projective} we give optimal bounds for Desarguesian projective planes. 
We have not considered  non-Desarguesian planes. 
\begin{proof}
  Let $\Gamma$ be a generalized $m$--gon of order $q$. Since the girth of $\Gamma$ equals $2m$, every embedded path of length at most $m$ is a geodesic, and it is the unique geodesic between its endpoints if its length is strictly less than $m$. 
  Furthermore, we claim that every embedded path of length at most $m+1$ can be completed to an induced $2m$--gon.  
  Indeed, suppose $x_0$, $x_1$,\dots, $x_i, x_{i+1}$ are consecutive vertices along an embedded edge path in $\Gamma$ for some $i\leq m$. 
  If $i<m$ then $x_{i+1}$ cannot be adjacent to $x_j$ for $j<i$ without violating the girth condition, but $x_{i+1}$ has valence $q+1\geq 3$, so it has some neighbor that we have not visited yet. 
  Thus, it is possible to extend the embedded path, and we may assume $i\geq m$.
  Since $\Gamma$ is bipartite, $d(x_0,x_j)\neq d(x_0,x_{j+1})$ for all $j$. 
  While distances to $x_0$ increase we are following a geodesic segment. 
  The diameter condition says this happens for at most $m$ steps, so there is a first index $j\leq m$  at which $d(x_0,x_{j+1})<d(x_0,x_j)$.  
  Then there is a geodesic from $x_{j+1}$ to $x_0$ of length $j-1$, so it cannot go through $x_j$ or $x_{j-1}$. 
  Thus, there is a loop $\gamma$ containing $x_0$, $x_{j-1}$, $x_j$, and $x_{j+1}$ of length $2j$ that uses the edge $x_{j-1}\edge x_j$ and the edge $x_j\edge x_{j+1}$ only once each.
  Tighten it to an embedded cycle $\gamma'$, which is possible since some edges of $\gamma$ are only used once.
  This contradicts the girth $2m$ condition unless $i=j=m$ and $\gamma=\gamma'$.

    Let $u\edge v$ be an edge of $\Gamma$, and let $u_1$,\dots,$u_{q}$ be neighbors of $u$ distinct from each other and $v$. For any choice of neighbor $v_1$ of $v$, distinct from $u$, there exists an induced $2m$--gon $H_1$ containing the segment $u_1\edge u\edge v\edge v_1$. 

    We proceed by induction. Suppose we have chosen $2m$--gons $H_i$ containing $u_i\edge u\edge v\edge v_i$ with $v_i$ distinct for all $i<k$, such that the $H_i$ pairwise
intersect in the edge $u\edge v$ and their union is induced.
Let $v_{k,1}$,\dots,$v_{k,q+1-k}$ be neighbors of $v$ distinct from
$u$ and the $v_i$ for $i<k$.
There are at least $q^{m-2}$ embedded paths $p$ of length $m-2$ starting from
$u_k$ and not containing $u$.
For each such path and each choice of $v_{k,j}$, there is an embedded path
$(v_{k,j}\edge v\edge u\edge u_k)+p$ of length $m+1$, which can be completed to a
$2m$--gon.
This gives at least $q^{m-2}(q-(k-1))$ potential choices of a $2m$--gon $H_k$ containing
$u_k\edge u\edge v$ and intersecting each $H_i$ only in $u\edge v$.

We now consider when such a choice of $H_k$ as above yields an induced union $\bigcup_{i=1}^k H_i$ by applying \Cref{lem:usually_induced}. 
Suppose $H_i=a_0,\dots,a_{2m-1}$ with $a_0=u$ and $a_1=v$ for some $i<k$. 
Suppose $B=b_0, \ldots, b_{2m-1}$ and $C=c_0, \ldots, c_{2m-1}$ are two different choices of $H_k$ with $b_{2m-1}=c_{2m-1}=u_k$, $b_0=c_0=u$ and $b_1=c_1=v$. 
By \Cref{lem:usually_induced}, to determine if $H_i \cup H_k$ is induced, it suffices to consider vertices adjacent to either $a_m$ or $a_{m-1}$. 
Because the girth of $\Gamma$ is $2m$, the vertex $a_m$ cannot be adjacent to both $b_{m+1}$ and $c_{m+1}$. 
Thus, among the choices for $H_k$ there is at most one that is bad because it contains a vertex adjacent to $a_m$. 
Again, as the girth of $\Gamma$ is $2m$, the vertex $a_{m+1}$ can be adjacent to both $b_m$ and $c_m$ only if $b_2\neq c_2$. 
So, for each of the $q+1-k$ choices of $v_{k,j}$ there is at most one choice of $H_k$ that is bad because it contains a vertex adjacent to $a_{m+1}$. 
Thus, there are at most $(q+1-k)+1$ bad choices of $H_k$ such that $H_i\cup H_k$ is not induced. 
As $i$ varies, this means the number of good choices of  $H_k$ is at least $q^{m-2}(q-(k-1))-(k-1)(q+1-(k-1))$. 
To successfully choose $H_k$ we need this number to be positive.

Define $f(x):=q^{m-2}(q-x)-x(q+1-x)$.
When $m>3$ the function $f$ is decreasing on the interval $[0,q-1]$
and positive at $q-1$, so it is possible to make good choices of $H_k$
as long as $k-1\leq q-1$.
Suppose $m=3$, so that $f(x)=x^2-(2q+1)x+q^2$.
  Then $f$ is positive on the interval $[0,q-\sqrt{q}]$, so it is
  possible to make good choices of $H_k$ as long as $k-1\leq\lfloor q-\sqrt{q}\rfloor$.

  The second case in \Cref{defn:n6_branching}, finding $2m$--gons
  sharing a segment of length 2 and whose union is induced, is easy. 
Consider a segment $u\edge w\edge v$ and any neighbors $u_1$,\dots,$u_q$ of $u$,
distinct from each other and $w$.
Choose neighbors $v_1$,\dots,$v_q$
of $v$, distinct from each other and $w$.
For each $i$, the path $u_i\edge u\edge w\edge v\edge v_i$ has length $4\leq m+1$, so it
can be completed to a $2m$--gon $H_i$.
By construction, $u\edge w\edge v$ is a maximal subsegment in the pairwise
intersection of the $H_i$, so by \Cref{lem:usually_induced} it is the
entire intersection and the union of the $H_i$ is induced.
\end{proof}

\begin{remark}
Generalized $m$--gons have an extensive history. 
Right-angled Coxeter groups with presentation graph a finite
generalized $m$--gon have been studied before.
Bounds and Xie \cite{BoundsXie} proved quasi-isometric rigidity for
such groups.
We refer to their paper for further reference on the history of
generalized $m$--gons.
In particular, finite, thick ($q\geq 2$) generalized $m$--gons exist only for $m\in\{2,3,4,6,8\}$, and for $m > 3$ there are conditions on the vertex degree (though examples of arbitrarily high degree always exist) \cite{feithigman}. 
          
The method of proof for the rigidity theorem of Bounds and Xie is to show that such a Coxeter group has the
structure of a Fuchsian building and then apply a rigidity theorem
of Xie \cite{Xie06} for such buildings.
Bourdon \cite{bourdon} computed the conformal dimension of the
boundary of these buildings, in the constant order case, to be
$1+\frac{\log(q)}{\log(\lambda_m)}$, where $\lambda_m$ is the
exponential growth rate of the underlying Coxeter group of the
building, which depends only on $m$ and is computable by
Floyd and Plotnick's \cite{MR877003} growth series
methods or by Steinberg's formula \cite{MR230728}.
\Cref{prop:branching_for_genm} and \Cref{thm:confdim_lower_bounds}
combine to give a coarser, but elementary, lower bound of (in the
$m>3$ case) $1+\frac{\log(q)}{\log(6m-7)}$, which, in terms of $q$,
has the correct logarithmic growth but smaller coefficient depending
on $m$. 
\end{remark}

\subsubsection{Projective planes}\label{sec:projective}
In this subsection we consider a special case of generalized 3--gons for which we can get a better branching estimate than the one provided by \Cref{prop:branching_for_genm}.

\begin{defn}\label{def:Levi_projective}
        Let $\LeviP_q$ denote the Levi graph for the points and lines of
        the projective plane
        over $\F_q$.
        These can be parameterized as 1 and 2--dimensional vector
        subspaces of the vector space $V=\F_q^3$.
    The incidence relation corresponds to containment of vector subspaces. 
  \end{defn}

    The space $\LeviP_q$ is a $1$-dimensional spherical building of
    type $A_2$ as described by Ronan~\cite[Chapter~1, Example 4]{ronan}. 
    An apartment in $\LeviP_q$ contains six chambers, each homeomorphic to an edge, that are arranged in a circuit. 
    In the projective plane over $\F_q$ each point is incident to $q+1$ lines and each line contains $q+1$ points. 
    In particular, $\LeviP_q$ is $(q+1)$--valent.
   The girth of $\LeviP_q$ is exactly 6, since the apartment gives a
   cycle of length 6, while our general remark on incidence structures
   implies the girth is at least 6.
  The diameter of $\LeviP_q$ is 3, and for every vertex there is some
  vertex at distance 3 from it, since every vertex is at distance 2
  from every other vertex of the same type, but for each line there is
  a plane not containing it, and vice versa.
  Thus $\LeviP_q$ is a generalized 3--gon of order $q$. 

   \begin{remark}
In addition to Bounds and Xie \cite{BoundsXie}, Schesler and Zaremsky
\cite{SchZar23} have considered right-angled Coxeter groups presented by Levi graphs of finite projective planes.
They showed that for all sufficiently large primes $p$ the commutator subgroup of $W_{\LeviP_p}$ algebraically fibers with kernel that is finitely generated but not of type $\mathrm{FP}_2$. 
    \end{remark}

    The group $GL_3(\F_q)$ acts on the vector space $V$ preserving dimension and incidence and can thus be identified with a subgroup of $\Aut(\LeviP_q)$ that preserves vertex type.
    There is also an involution in $\Aut(\LeviP_q)$ that exchanges points and lines, coming from projective duality.

    \begin{prop} \label{prop:n6_examples}
The Levi graph $\LeviP_q$ of the projective plane over $\F_q$ is a finite, simplicial graph with diameter 3 and girth 6.
It has $(1,6)$--branching if $q=2$ and $(q,6)$--branching otherwise.      \end{prop}
\begin{proof}
We have explained diameter and girth. We will check the conditions of \Cref{defn:n6_branching}. 

Fix a basis $e_1 = (1,0,0)$, $e_2 = (0,1,0)$, $e_3 = (0,0,1)$ for $V =
\F_q^3$.

Suppose $P=u\edge w\edge v$ is an induced edge path in $\LeviP_q$. 
Up to the action of $\Aut(\LeviP_q)$, we may assume that $u = \la e_1, e_2 \ra$, $w = \la e_1 \ra$, and $v = \la e_1, e_3 \ra$.
There are $q^2$--many hexagons containing $P$, parameterized by
$(i,j)\in \F_q^2$ as in \Cref{fig:hexagon2}.

\begin{figure}[h]
    \centering
    \begin{tikzpicture}[every label/.append style={overlay}]\small
        \path[use as bounding box] (-1.2,-1.25) rectangle (1.2,1.25);
        \coordinate[label={[label distance=-1pt] 180:$u=\langle e_1,\, e_2\rangle$}] (u) at (-150:1);
       \coordinate[label={[label distance=-1pt] -90:$w=\langle e_1\rangle$}] (w) at (-90:1);
       \coordinate[label={[label distance=-1pt] 0:$v=\langle e_1,\, e_3\rangle$}] (v) at (-30:1);
       \coordinate[label={[label distance=-1pt] 0:$v_j:=\langle je_1+e_3\rangle$}] (vj) at (30:1);
       \coordinate[label={[label distance=-1pt] 90:$x_{i,j}:=\langle ie_1+e_2,\, je_1+e_3\rangle$}] (ijk) at (90:1);
      \coordinate[label={[label distance=-1pt] 180:$u_i:=\langle ie_1+e_2\rangle$}] (ui) at (150:1);
      \draw (u)--(ui)--(ijk)--(vj)--(v)--(w)--(u);
    \end{tikzpicture}
    \caption{Hexagon $H(i,j)$ containing segment $u\edge w\edge v$.}
    \label{fig:hexagon2}
    \figurealttext{A hexagon labeled $H(i,j)$ containing the
      length-two path $P=u\edge w\edge v$. Moving cyclically around
      the hexagon from $u$, the vertices are $u$, $w$, $v$, $v_j$,
      $x_{i,j}$, and $u_i$. The labels are $u=\langle
      e_1,e_2\rangle$,$w=\langle e_1\rangle$, $v=\langle
      e_1,e_3\rangle$, $v_j=\langle je_1+e_3\rangle$, $x_{i,j}=\langle ie_1+e_2,je_1+e_3\rangle$, and $u_i=\langle ie_1+e_2\rangle$.}
  \end{figure}
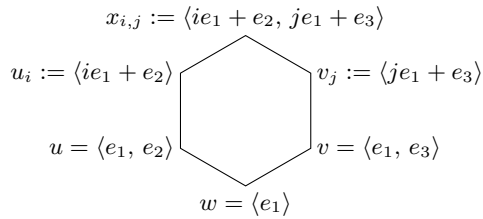

Consider the hexagons $H(i,i)$ for $i\in \F_q$.
Their pairwise intersection is $P$, and   $\bigcup_{i\in\F_q}H(i,i)$ is an induced subgraph by \Cref{lem:usually_induced}. 
        
Now suppose $u\edge v$ is a single edge of $\LeviP_q$.
Up to the action of $\Aut(\LeviP_q)$ we may assume $u = \la e_1 \ra$ and $v = \la e_1, e_2 \ra$.
\Cref{fig:hexagon1} shows distinct hexagons containing this edge, parameterized by $(i,j,k)\in \F_q^3$.

  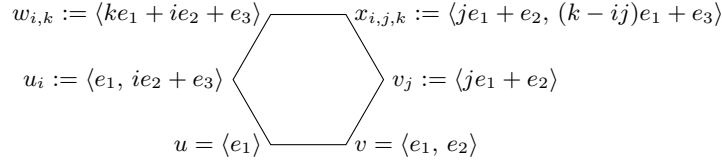
\begin{figure}[h]
    \centering
    \begin{tikzpicture}[every label/.append style={overlay}]\small
      \path[use as bounding box] (-1.2,-1.25) rectangle (1.2,1.25);
      \coordinate[label={[label distance=0pt] 180:$u=\langle e_1\rangle$}] (u) at (-120:1);
      \coordinate[label={[label distance=0pt] 0:$v=\langle e_1,\, e_2\rangle$}] (v) at (-60:1);
      \coordinate[label={[label distance=0pt] 180:$u_i:=\langle e_1,\, ie_2+e_3\rangle$}] (ui) at (180:1);
      \coordinate[label={[label distance=0pt] 0:$v_j:=\langle je_1+e_2\rangle$}] (vj) at (0:1);
      \coordinate[label={[label distance=0pt] 0:$x_{i,j,k}:=\langle je_1+e_2,\, (k-ij)e_1+e_3\rangle$}] (ijk) at (60:1);
      \coordinate[label={[label distance=0pt] 180:$w_{i,k}:=\langle ke_1+ie_2+e_3\rangle$}] (k) at (120:1);
      \draw (u)--(v)--(vj)--(ijk)--(k)--(ui)--(u);
    \end{tikzpicture}
    \caption{Hexagon $H(i,j,k)$ containing edge $u\edge v$.}
    \label{fig:hexagon1}
    \figurealttext{A hexagon labeled $H(i,j,k)$ containing the edge
  $u\edge v$. Moving cyclically around the hexagon from $u$, the
  vertices are $u$, $v$, $v_j$, $x_{i,j,k}$, $w_{i,k}$, and $u_i$. The
  vertex labels are given by the displayed subspaces: $u=\langle e_1\rangle$, $v=\langle e_1,e_2\rangle$, $v_j=\langle je_1+e_2\rangle$, $x_{i,j,k}=\langle je_1+e_2,(k-ij)e_1+e_3\rangle$, $w_{i,k}=\langle ke_1+ie_2+e_3\rangle$, and $u_i=\langle e_1,ie_2+e_3\rangle$.}
  \end{figure}

Suppose that $i_1\neq i_2$ and $j_1\neq j_2$, and consider
hexagons of the form $H(i_1,j_1,k_1)$ and $H(i_2,j_2,k_2)$.
By \Cref{lem:usually_induced}, the only possibility for $H(i_1,j_1,k_1)\cup H(i_2,j_2,k_2)$ not to be an induced subgraph is if, in the notation of \Cref{fig:hexagon1}, the vertex $w_{i_1,k_1}$ is adjacent to $x_{i_2,j_2,k_2}$ or $w_{i_2,k_2}$ is adjacent to $x_{i_1,j_1,k_1}$.
Each of these possibilities gives a system of linear equations that have solutions exactly when   $k_2=k_1+j_1(i_2-i_1)$ or $k_2=k_1+j_2(i_2-i_1)$.

$\LeviP_2$ is the Heawood graph.
For every pair of hexagons in $\LeviP_2$ sharing a single edge, the union is not an induced subgraph.
        
Suppose $q>2$.
First consider the hexagons $H(i,i^{-1},0)$ for $i\in\F_q^*$.
This is a collection of $q-1$ hexagons with the correct intersections, and their union is induced, since for any $i_1\neq i_2$ we have $j_1=i_1^{-1}\neq j_2=i_2^{-1}$ are non-zero, while $\frac{k_2-k_1}{i_2-i_1}=0$.

Since $q\neq 2$ there exists $k\in\F_q^*\setminus\{-1\}$.
Add the hexagon $H(0,0,k)$ to our collection.
Our conditions imply that the union is still induced unless $k=0$ or there is an $i\in \F_q^*$ such that $k=-ii^{-1}=-1$, both of which contradict the choice of $k$.
\end{proof}

\begin{corollary}\label{prop:levi_projective_plane_menger}
$W_{\LeviP_q}$ is hyperbolic and $\bdry W_{\LeviP_q}$ is the
Menger curve, with conformal dimension bounded below by a function
that goes to infinity with $q$. 
\end{corollary}
\begin{proof}
This follows from \Cref{prop:n6_examples},
\Cref{prop:branching_usually_implies_menger}, and
\Cref{thm:confdim_lower_bounds}, except that for $q=2$ the branching
condition is not strong enough to deduce inseparability and nonplanarity.
However, $\LeviP_2$ is the Heawood graph, which is inseparable and nonplanar.
\end{proof}

\subsection{Transversal designs}\label{sec:td}
\begin{defn}
  A \emph{transversal design} $\td(t,m)$ consists of $tm$--many \emph{Points}
  partitioned into $t$--many \emph{Parts} each of size $m$ and a collection of
  subsets of Points, called \emph{Blocks}, such that each Block
  consists of exactly one Point from each Part and every pair of
  Points from distinct Parts belongs to exactly one Block.
  The Points and Blocks give an incidence structure with respect to
  containment. 
\end{defn}

Transversal designs start to get interesting for us when
$t,m\geq 3$.
There do not exist transversal designs for every choice of parameters
$t,m$, but there are examples known for $\td(3,m)$
(\emph{Latin squares}) and for $\td(t,q)$ with $t\leq q+1$, where $q$ is a prime power.
See \Cref{sec:affine} for existence of $\td(q+1,q)$. The existence of
$\td(t,q)$ for $t\leq q$ follows from the equivalence between $\td(t,q)$ and a collection of $t-2$ mutually orthogonal Latin squares \cite{MR890103}.
It is not hard to show that, for $t,m>1$, a necessary
condition for the existence of a $\td(t,m)$ is given by $t\leq m+1$. 
The parameters $t,m$ do not, in general, determine a unique
transversal design nor a unique isomorphism type of Levi graph.

As shown in \Cref{prop:branching_for_td}, the Levi graph of a transversal design with $t,m \geq 3$ has girth 6 and diameter 4. Since the girth is not twice the diameter, the Levi graph of $\td(t,m)$ is not a generalized $n$--gon, so these yield different examples than the ones in \Cref{sec:genm}. 

\begin{prop}\label{prop:branching_for_td}
  For $t,m\geq 3$, if there exists a transversal design $\td(t,m)$
  then its Levi graph $\LeviT_{t,m}$ is a finite, simplicial graph
  with diameter 4, girth 6, and $(\max\{2,n\},6)$--branching for $n:=\min\{t,m\}-2$. 
\end{prop}
\begin{proof}
    We first verify the girth and diameter conditions. The diameter of $\LeviT_{t,m}$ is at least 4, because Points from the same Part are not contained in a common Block, so their distance is greater than two, and their distance is even because the graph is bipartite. 
    The distance between a Point $x$ and a Block $B$ is 1 if $x\in B$ and 3 if $x\notin B$ (in the latter case, a length 3 path from $B$ to $x$ is obtained by stepping from $B$ to a Point $y\in B$ from a different part than $x$, then to the unique Block containing both $x$ and $y$, then to $x$). As every Point is 1 away from some Block and vice versa, it follows that the distance between a pair of Points or a pair of Blocks is at most 4.  Thus, the diameter is exactly 4. 
    The girth is at least 6 because it is even, but it cannot be 4 because there is at most one Block containing any pair of Points. The girth then equals 6 because, as we will see, there are lots of hexagons.
    
    We now verify the branching conditions of \Cref{defn:n6_branching}. Suppose the Points are numbered $(i,j)$ for $0\leq i\leq t-1$ and $0\leq j\leq m-1$, where the partition is by the value of the first coordinate.
    Each Point is contained in $m$ Blocks, and each Block contains $t$
    Points, so the valence of each vertex of $\cT_{t,m}$ is at least $\min\{t,m\} \geq n+1$.  

    First suppose the induced path $P$ is an edge in $\LeviT_{t,m}$. Up to renumbering, assume it consists of Point $(0,0)$ and Block $B_0$. The argument has separate cases for $n\leq 2$ and $n>2$. 
    
    Suppose $n\leq 2$. The Block $B_0$ contains one Point from each Part.
  Choose any two of its neighbors different from $(0,0)$.
  Up to renumbering, assume they are $(1,0)$ and $(2,0)$. The Point $(0,0)$ belongs to $m$ Blocks. Choose any two of its neighbors different from $B_0$, and call them $B_1$ and $B_2$.

Each Block contains one Point from each part.
Up to renumbering, we may assume the Point from part 2 in $B_1$ is
$(2,1)$, and the Point from part 1 in $B_2$ is $(1,1)$.
There is a unique Block $B_1'$ containing Points $(1,0)$ and $(2,1)$,
and there is a unique Block $B_2'$ containing Points $(1,1)$ and
$(2,0)$.
This gives a pair of hexagons sharing the base edge as in
\Cref{fig:two_hex_for_latin_squares}.
Their union is induced because $B_1'$ contains Point $(1,0)$ from Part
1, so it cannot contain $(1,1)$, and similarly $B_2'$ contains $(2,0)$
from Part 2, so it cannot contain $(2,1)$. This yields the two desired cycles, pictured in \Cref{fig:two_hex_for_latin_squares}.

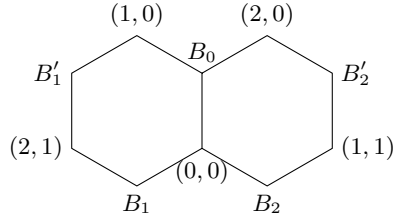
\begin{figure}[h]
  \centering
  \begin{tikzpicture}[rotate=90]\small
             \coordinate[label={[label distance=1pt] -90:$(0,0)$}] (u) at (0,0);
            \coordinate[label={[label distance=1pt] 90:$B_0$}] (v) at (1,0);
            \coordinate[label={[label distance=0pt] -90:$B_1$}] (u1) at (-.5,{sqrt(3)/2});
            \coordinate[label={[label distance=0pt] 90:$(1,0)$}] (v1) at (1.5,{sqrt(3)/2});
            \coordinate[label={[label distance=0pt] 180:$(2,1)$}] (w1) at (0,{sqrt(3)});
            \coordinate[label={[label distance=0pt] 180:$B_1'$}] (x1)
            at (1,{sqrt(3)});
            \coordinate[label={[label distance=0pt] -90:$B_2$}] (u2) at (-.5,{-sqrt(3)/2});
            \coordinate[label={[label distance=0pt] 90:$(2,0)$}] (v2) at (1.5,{-sqrt(3)/2});
            \coordinate[label={[label distance=0pt] 0:$(1,1)$}] (w2) at (0,{-sqrt(3)});
            \coordinate[label={[label distance=0pt] 0:$B_2'$}] (x2) at (1,{-sqrt(3)});
            \draw (u)--(u1)--(w1)--(x1)--(v1)--(v)--(u)--(u2)--(w2)--(x2)--(v2)--(v);
          \end{tikzpicture}
  \caption{A pair of independent hexagons in $\LeviT_{t,m}$ sharing one edge.}
  \label{fig:two_hex_for_latin_squares}
  \figurealttext{Two hexagons in the Levi graph of $\td(t,m)$ sharing the base edge between the Point $(0,0)$ and the Block $B_0$. One hexagon passes through $B_1$, $(2,1)$, $B_1'$, and $(1,0)$ before returning to $B_0$. The other passes through $B_2$, $(1,1)$, $B_2'$, and $(2,0)$ before returning to $B_0$. The figure illustrates two induced hexagons sharing exactly one edge.}
\end{figure}

When $n>2$ we make  different choices.
Block $B_0$ has $t$--many neighbors, one Point from each of the Parts.
Choose any $n<t-1$ of them, distinct from each other and from
$(0,0)$, and assume, up to renumbering that they are
$(1,0)$,\dots,$(n,0)$. The Point $(0,0)$ belongs to $m$--many Blocks. Choose any $n<m-1$ of them, distinct from each other and from $B_0$, and call them
$B_1$,\dots,$B_n$. 
For each $1\leq i\leq n$, suppose that Point $(n+1,i)$ is the Point
from Part $n+1$ contained in $B_i$, and let $B_i'$ be the unique Block containing $(n+1,i)$
and $(i,0)$.
Let $H_i$ be the hexagon with vertex set
$\{(0,0),B_i,(n+1,i),B_i',(i,0),B_0,(0,0)\}$, as in
\Cref{fig:td_hexagons}.

\begin{figure}[h]
  \centering
  \begin{tikzpicture}\tiny
             \coordinate[label={[label distance=0pt,xshift=-4pt] -45:$(0,0)$}] (u) at (0,0);
            \coordinate[label={[label distance=0pt,xshift=3pt] -135:$B_0$}] (v) at (1,0);
            \coordinate[label={[label distance=0pt] 180:$B_1$}] (u1) at (-.5,{sqrt(3)/2});
            \coordinate[label={[label distance=0pt,xshift=1pt] 180:$(1,0)$}] (v1) at (1.5,{sqrt(3)/2});
            \coordinate[label={[label distance=0pt] 180:$(n+1,1)$}] (w1) at (0,{sqrt(3)});
            \coordinate[label={[label distance=0pt] 0:$B_1'$}] (x1)
            at (1,{sqrt(3)});
            \coordinate[label={[label distance=0pt] 180:$B_2$}] (u2) at (.25,{-.25+sqrt(3)/2});
            \coordinate[label={[label distance=0pt] 0:$(2,0)$}] (v2) at (2,{-.25+sqrt(3)/2});
            \coordinate[label={[label distance=-4pt] 135:$(n+1,2)$}] (w2) at (1,{sqrt(3)-.5});
            \coordinate[label={[label distance=0pt] 0:$B_2'$}] (x2) at
            (2,{sqrt(3)-.5});
            \coordinate[label={[label distance=0pt] 180:$B_n$}] (un) at (-.5,{-sqrt(3)/2});
            \coordinate[label={[label distance=0pt] 0:$(n,0)$}] (vn) at (1.5,{-sqrt(3)/2});
            \coordinate[label={[label distance=0pt] 180:$(n+1,n)$}] (wn) at (0,{-sqrt(3)});
            \coordinate[label={[label distance=0pt] 0:$B_n'$}] (xn) at
            (1,{-sqrt(3)});
            \coordinate[label={[label distance=0pt] 180:$B_{n+1}$}] (unp) at (-.75,0);
            \coordinate[label={[label distance=0pt] 0:$(n+1,0)$}] (vnp) at (1.75,0);
            \draw (u)--(u1)--(w1)--(x1)--(v1)--(v)--(u);
            \draw[ultra thick,white] (w2)--(x2);
            \draw (u)--(u2)--(w2)--(x2)--(v2)--(v)--(vn)--(xn)--(wn)--(un)--(u);
            \draw (u)--(unp) (v)--(vnp);
          \end{tikzpicture}
  \caption{$n$--many independent hexagons in $\LeviT_{t,m}$ when $n>2$.}
  \label{fig:td_hexagons}
  \figurealttext{Diagram of $n$ independent hexagons in the Levi graph
    of $\td(t,m)$ sharing the base edge between the Point $(0,0)$ and the Block $B_0$. The hexagon $H_i$ passes through $B_i$, $(n+1,i)$, $B_i'$, and $(i,0)$ before returning to $B_0$. The figure shows representative hexagons $H_1$, $H_2$, and $H_n$, together with additional unused neighbors $B_{n+1}$ and $(n+1,0)$, illustrating that the chosen hexagons intersect exactly in the base edge.}
\end{figure}

By construction, the pairwise intersection of distinct hexagons from this
collection is exactly the edge $\{(0,0),B_0\}$.
To see that their union is induced, observe that $(n+1,i)$ is the only
point from part $n+1$ in $B_i'$, so there is no edge from $(n+1,j)$ to
$B_i'$ when $i\neq j$. Thus, there exist $\max\{2,n\}$--many hexagons that pairwise
intersect in the edge $\{(0,0),B_0\}$ and whose union is an induced
subgraph.

    Now suppose the induced path $P$ has length two. This case is easier: the union is always induced, by \Cref{lem:usually_induced}, so one only needs to find enough hexagons.
    This is easy because any choice of 3 Points from distinct Parts are either contained in a single Block or yield a unique hexagon. 
     So if $P$  is of the form    Point-Block-Point, then we can assume, up to renumbering, that it is $(1,0)\edge B_0\edge (2,0)$, and that the Point in $B_0$ from Part 3 is $(3, 0)$.  So $B_0$ does not contain the Points 
 $(3, 1), \dots, (3, m-1)$, and together with $(1, 0)$ and $(2,0)$, these form $m-1 \ge \max\{2, n\}$ distinct hexagons whose pairwise intersection is $P$.  
 If the segment is of the form Block-Point-Block, we may assume, up to renumbering, that it is $B_1\edge (0,0)\edge B_2$, and that for $1 \le i \le m-1$,  the points in $B_1$ and $B_2$ from Part $i$ are $(i, 1)$ and $(i,2)$ respectively.  
 Then for any choice of fixed-point-free permutation $\sigma$ of $\{1,\dots,m-1\}$ we get $m-1$ distinct hexagons pairwise intersecting in $P$ by considering the triples $(0, 0), (i, 1), (\sigma(i), 2)$.
\end{proof}

\subsubsection{Affine planes}\label{sec:affine}
\begin{defn}
  The \emph{affine plane over $\F_q$} consists of \emph{points}
  $(x,y)$ for $x,y\in\F_q$ and lines:
  \begin{itemize}
      \item $[m,b]$ for $m,b\in\F_q$. \quad(slope $m$ lines)
      \item $[\infty,a]$ for $a\in\F_q$. \quad(vertical/slope $\infty$ lines)
  \end{itemize}
  Point $(x,y)$ lies on a vertical line $[\infty,a]$ if and only if $x=a$. Point $(x,y)$ lies on a non-vertical line $[m,b]$ if and only if $y=mx+b$.
 Let $\LeviA_q$ be the Levi graph of this incidence structure.
\end{defn}

$\LeviA_q$ has $q^2$ points, each contained in $(q+1)$--many lines, and $(q^2+q)$ lines, each containing $q$--many points. 
Every pair of distinct points is contained in a common line, but
distinct lines with the same slope do not intersect.

The affine plane can alternatively be thought of as a sub-incidence
structure of the projective plane, with the points at infinity and the
line at infinity removed.

Consider the transversal design whose Points are affine lines,
partitioned into Parts by slope, and taking Blocks to be the affine points.
Each affine point lies on exactly one line of each slope, and every
two lines with different slopes intersect in a unique point. So from
the affine plane over $\F_q$ we get a $\td(q+1,q)$.  Observe that $\LeviA_q$ is the Levi graph of this transversal design. 
From \Cref{prop:branching_for_td} we can deduce that $\LeviA_q$ has
$(\max\{2,q-2\},6)$--branching, but actually we can do better:
\begin{prop}\label{prop:branching_for_affine_plane}
Let $q>2$ be a prime power. The Levi graph $\LeviA_q$ of the affine plane over $\F_q$ is a finite, simplicial graph with diameter 4, girth 6, and $(q-1,6)$--branching. 
\end{prop}
We  skip this proof since  \Cref{prop:n6_examples_biaffine} is very similar. 
There is an analogue of \Cref{prop:levi_projective_plane_menger}:
\begin{corollary}
  $W_{\LeviA_q}$ is hyperbolic.
  Its boundary is a `tree of graphs' if $q=2$. Otherwise, $\bdry
  W_{\LeviA_q}$ is the Menger curve, with conformal dimension bounded below by a function that goes to infinity with $q$.
\end{corollary}
\begin{proof}
$\LeviA_2$ is a subdivided $K_4$, which is not inseparable: pairs of
essential vertices are cut pairs, so $W_{\LeviA_2}$ has a non-trivial
JSJ decomposition over 2--ended subgroups and its boundary is a tree of graphs as in \cite{HodSwi25}.         
For $q>2$ the result  follows from
\Cref{prop:branching_for_affine_plane},
\Cref{prop:branching_usually_implies_menger},  and
\Cref{thm:confdim_lower_bounds}, except that nonplanarity of $\LeviA_3$ is not forced by branching.
However,  $\LeviA_3$ is the Hesse graph, which is nonplanar.
\end{proof}

\subsubsection{Biaffine planes}\label{sec:biaffine}
\begin{defn}
  The \emph{biaffine plane over $\F_q$} consists of \emph{points}
  $(x,y)$ for $x,y\in\F_q$ and lines $[m,b]$ for $m,b\in\F_q$.
  Point $(x,y)$ lies on line $[m,b]$ if and only if $y=mx+b$.
 Let $\LeviB_q$ be its Levi graph.
\end{defn}

This biaffine plane\footnote{This terminology follows \cite{AraLee22}. There are other incidence
  structures called `the biaffine plane' in the literature.} is the sub-incidence structure obtained by removing the vertical lines from the affine plane. 
Any other pair of points is contained in a line.
Symmetrically, `parallel' lines $[m,b_1]$ and $[m,b_2]$ with $b_1\neq b_2$ have no point
in common, but any other pair of lines intersects in a point. Every vertex of $\LeviB_q$ has valence $q$.

For any choice of $\alpha,\gamma\in \F_q^*$ and $\beta,\delta\in \F_q$
there is a type-preserving automorphism $\phi_{\alpha,\beta,\gamma,\delta}$ of
$\LeviB_q$:
\begin{align*}
  (x,\,y)&\mapsto (\alpha x+\beta,\, \gamma y+\delta)\\
  [m,\,b]&\mapsto [\frac{\gamma m}{\alpha},\,\delta+\gamma
           b-\frac{\beta\gamma m}{\alpha}]
\end{align*}
There is also a type-reversing involution $\iota$:
\begin{align*}
  (x,y)&\mapsto [x,-y]\\
  [m,b]&\mapsto (m,-b)
\end{align*}
These are enough automorphisms to show that $\LeviB_q$ is vertex-transitive.

As in the affine case, we can make a transversal design from the
biaffine plane by taking the Points to be the lines of the biaffine plane,
partitioned into Parts by slope, and taking the Blocks to be biaffine
points. This yields a $\td(q,q)$.
Since this incidence structure is self-dual, there is an equivalent
$\td(q,q)$ obtained by taking the Points to be the affine points,
where the partition is by vertical lines, and taking the Blocks to be
the non-vertical lines.
However, the explicit structure of the biaffine plane allows us to
derive a better branching condition than direct application of 
\Cref{prop:branching_for_td}.
\begin{prop} \label{prop:n6_examples_biaffine}
The Levi graph $\LeviB_q$ of the biaffine plane over $\F_q$ is a
finite, simplicial, vertex-transitive graph with diameter 4.
The graph $\LeviB_2$ is an octagon. For $q>2$, $\LeviB_q$ has girth 6 and  $(q-1, 6)$-branching.  
\end{prop}
\begin{proof}
The graph $\LeviB_2$ can be checked by explicit construction, so assume $q>2$.
The automorphism group of $\LeviB_q$ is transitive on vertices and the
stabilizer of $(0,0)$ is transitive on its neighbors, so it suffices to suppose that the base edge is $(0,0)\edge [0,0]$.
The other neighbors of $(0,0)$ are $[x_i,0]$ where $x_i$ ranges over the elements of $\F_q^*$.
The other neighbors of $[0,0]$ are $(y_i,0)$ where $y_i$ ranges over the elements of $\F_q^*$.
We will assume that the indexing of one of these sets is given, and
the other one can be chosen arbitrarily except for the condition that
$x_iy_i\neq 1$, which is possible since $q>2$.

For any choice of $z_i\in\F_q^*$ with $z_i\neq y_i$ we can find a hexagon $H_i$ as in \Cref{fig:biaffine_hexagon}.
For our purposes, choose $z_i=x_i^{-1}$ for all $i$. 
Then, to see that $\bigcup_{1\leq i<q}H_i$ is induced, suppose
$(z_j,x_jz_j)\in [\frac{x_iz_i}{z_i-y_i},\,\frac{-x_iy_iz_i}{z_i-y_i}]$.
Having chosen $z_i=x_i^{-1}$ and $z_j=x_j^{-1}$, this yields:
\[1=x_j z_j=\frac{x_i z_i}{z_i-y_i}\cdot z_j+\frac{-x_i y_i z_i}{z_i-y_i}=\frac{x_j^{-1}}{x_i^{-1}-y_i}-\frac{y_i} {x_i^{-1}-y_i}\]
Solve for $x_j$ to find $x_j=x_i$, which is a contradiction.

\begin{figure}[h]
  \centering
  \begin{tikzpicture}\small
    \coordinate[label={[label distance=0pt] 180:$(0,0)$}] (u) at (-120:1);
    \coordinate[label={[label distance=0pt] 0:$[0,0]$}] (v) at (-60:1);
    \coordinate[label={[label distance=0pt] 180:$[x_i,0]$}] (ui) at (180:1);
    \coordinate[label={[label distance=0pt] 0:$(y_i,0)$}] (vj) at (0:1);
    \coordinate[label={[label distance=0pt] 0:$[\frac{x_iz_i}{z_i-y_i},\,\frac{-x_iy_iz_i}{z_i-y_i}]$}] (ijk) at (60:1);
    \coordinate[label={[label distance=0pt] 180:$(z_i,x_iz_i)$}] (k) at (120:1);
    \draw (u)--(v)--(vj)--(ijk)--(k)--(ui)--(u);
  \end{tikzpicture}
  \caption{Hexagon $H_i$ in $\LeviB_q$ with base edge $(0,0)-[0,0]$.}
  \label{fig:biaffine_hexagon}
  \figurealttext{A hexagon labeled $H_i$ in the Levi graph $\LeviB_q$ containing the base edge $(0,0)\edge [0,0]$. Moving cyclically around the hexagon from $(0,0)$, the vertices are $(0,0)$, $[0,0]$, $(y_i,0)$, $[\frac{x_iz_i}{z_i-y_i},,\frac{-x_iy_iz_i}{z_i-y_i}]$, $(z_i,x_iz_i)$, and $[x_i,0]$. The figure illustrates one of the hexagons used in the $(q-1,6)$-branching argument.}
\end{figure}

Checking that there are many hexagons for the case that the base is a
segment of length 2 is easy, and, as in previous cases, we do not need
to choose very carefully to avoid the union being non-induced, since
this is a consequence of the girth 6 condition. 
\end{proof}

\begin{corollary}
  $W_{\LeviB_q}$ is hyperbolic.
  Its boundary is the circle if $q=2$. Otherwise, $\bdry W_{\LeviB_q}$ is the Menger curve, with conformal dimension bounded below by a function that goes to infinity with $q$.
\end{corollary}
\begin{proof}
The graph $\LeviB_2$ is an octagon, so $W_{\LeviB_2}$ is virtually a closed surface group and $\bdry W_{\LeviB_2}$ is a circle.         
For $q>2$  the result follows from \Cref{prop:n6_examples_biaffine},
\Cref{prop:branching_usually_implies_menger},  and
\Cref{thm:confdim_lower_bounds}, except that nonplanarity of $\LeviB_3$ is not forced by branching.
However, $\LeviB_3$ is the Pappus graph, which is nonplanar.
\end{proof}

\bibliographystyle{hypersshort}
\bibliography{Pontryagin_boundaries_references.bib}

\end{document}